\documentclass[11pt,reqno]{amsart}
\usepackage{amsmath,amssymb,amsthm,latexsym,cite,cancel}
\usepackage[small]{caption}
\usepackage{graphicx,wasysym,overpic,tikz,color}
\usepackage{subfigure}
\usepackage{cite}
\usepackage[colorlinks=true,urlcolor=blue,
citecolor=red,linkcolor=blue,linktocpage,pdfpagelabels,
bookmarksnumbered,bookmarksopen]{hyperref}
\usepackage[italian,english]{babel}
\usepackage{units}
\usepackage{enumitem}
\usepackage[left=2.1cm,right=2.1cm,top=2.71cm,bottom=2.71cm]{geometry}
\usepackage[hyperpageref]{backref}
\usepackage{float}

\usepackage[colorinlistoftodos]{todonotes}

\makeatletter
\def\namedlabel#1#2{\begingroup
    #2%
    \def\@currentlabel{#2}%
    \phantomsection\label{#1}\endgroup
}
\makeatother

\newtheorem{theorem}{Theorem}[section]
\newtheorem{definition}[theorem]{Definition}

\newtheorem{lemma}[theorem]{Lemma}
\newtheorem{remark}[theorem]{Remark}

\newcommand{\abs}[1]{\lvert#1\rvert}
\newcommand{\norm}[1]{\lVert#1\rVert}

\newcommand{\G}{\mathcal{G}}
\newcommand{\K}{\mathcal{K}}

\newcommand{\E}{\mathrm{E}}
\newcommand{\V}{\mathrm{V}}
\newcommand{\R}{\mathbb{R}}
\newcommand{\C}{\mathbb{C}}

\tikzstyle{nodo}=[circle,draw,fill,inner sep=0pt,minimum size=%
1.5mm]

\numberwithin{equation}{section}

\title[Nonlinear magnetic Schr\"odinger equations]{Normalized solutions of nonlinear magnetic Schr\"odinger equations on metric graphs}

\author[P. d'Avenia]{Pietro  d'Avenia}

\address[P. d'Avenia]{\newline\indent
	Dipartimento di Meccanica, Matematica e Management
	\newline\indent
	Politecnico di Bari
	\newline\indent
	Via Orabona 4,  70125  Bari, Italy}
\email{\href{mailto:pietro.davenia@poliba.it}{pietro.davenia@poliba.it}}

\author[Z. He]{Zhentao He}

\address[Z. He]{\newline\indent
	School of Mathematics
	\newline\indent
	East China University of Science and Technology
	\newline\indent
	Shanghai 200237, PR China }
\email{\href{mailto:hezhentao2001@outlook.com}{hezhentao2001@outlook.com}}

\author[C. Ji]{Chao Ji}

\address[C. Ji]{\newline\indent
	School of Mathematics
	\newline\indent
	East China University of Science and Technology
	\newline\indent
	Shanghai 200237, PR China }
\email{\href{mailto:jichao@ecust.edu.cn}{jichao@ecust.edu.cn}}


\begin{document}

\begin{abstract}
In this paper we first establish the theory of a magnetic Sobolev space $H^1_A(\mathcal{G},\mathbb{C})$ on metric graphs $\mathcal{G}$ and  prove the self-adjointness of the associated magnetic Schr\"odinger operator. Then, in this setting, we investigate the existence and multiplicity of normalized solutions to nonlinear magnetic Schr\"odinger equations on compact metric graphs and on noncompact metric graphs with localized nonlinearities or nonlinearities acting on whole metric graphs, covering the mass-subcritical, mass-critical, and mass-supercritical cases.
\end{abstract}

\maketitle

\setcounter{tocdepth}{3}
\begin{center}
	\begin{minipage}{12.5cm}
		\tableofcontents
	\end{minipage}
\end{center}

\section{Introduction}

In this paper, we consider a magnetic Schr\"odinger operator on connected metric graphs.

Quantum graphs (metric graphs equipped with differential operators) arise naturally as simplified models in mathematics, physics, chemistry, and engineering when one considers the propagation of waves of various kinds through a quasi-one-dimensional (e.g., meso- or nano-scale) system that looks like a thin neighborhood of a graph. For further details on quantum graphs, one may refer to \cite{Be}.

 Let $\mathcal{G}=(\mathrm{V}, \mathrm{E})$ be a connected metric graph, where $\mathrm{E}$ is the set of edges and $\mathrm{V}$ is the set of vertices. As usual, each bounded edge $e\in {\rm E}$ is identified with a closed and bounded interval $I_e = [0, \ell_e]$ with $\ell_e > 0$, while each unbounded edge $e$ is identified with 
 a closed half-line $I_e = [0, +\infty)$. If $\G$ is a metric graph with a finite number of edges, its compact core $\K$ is defined as the metric subgraph of $\G$ consisting of all bounded edges of $\G$. Therefore, if $\G$ is compact, then its compact core $\K = \G$. 

Throughout the paper, we assume that $\mathcal{G}$ has a finite number of edges and the compact core $\K$ is non-empty, unless otherwise specified.

A connected metric graph has the natural structure of a locally compact metric space, and for any two points $x,y \in \G$, the metric $\operatorname{dist}(x,y)$ is given by the shortest distance along the edges. Moreover, a connected metric graph with a finite number of vertices is compact if and only if it contains no half-line.
\subsection{Second-order nonlinear equations on metric graphs}

From a mathematical point of view, and, in particular from a variational point of view, several efforts have  been devoted to the existence and multiplicity of normalized solutions for the following nonlinear Schr\"odinger equation
\begin{equation}
\label{eqs}
    -u'' - \lambda u= \abs{u}^{p-2}u
\end{equation}
on metric graphs $\G$.\\

In the mass-subcritical ($2<p<6$) and mass-critical ($p=6$) cases, the energy functional associated with \eqref{eqs} is bounded below and coercive under the mass constraint (requiring the mass to be below a certain threshold when $p=6$). This allows us to use minimization methods to obtain normalized ground states, i.e., solutions with minimal energy under the mass constraint. The existence of such normalized ground states has been established both on compact metric graphs \cite{CDS,Do}, and on noncompact metric graphs \cite{ABD,ACFN,AST,AST2,AST3,Do1,NP,PS}. In addition, the existence of local minimizers was studied in \cite{PSV}.

For any noncompact metric graph $\G$ with a finite number of edges and a non-empty compact core $\K$, \cite{Gn,No} introduced the following modification of \eqref{eqs}
\begin{equation}
\label{eqsloc}
    -u'' - \lambda u= \chi_\K\abs{u}^{p-2}u,
\end{equation}
where $\chi_\K$ is the characteristic function of the compact core $\K$, assuming that the nonlinearity affects only the non-empty compact core $\K$. For the existence or nonexistence of normalized ground states for \eqref{eqsloc} in the mass-subcritical case, see \cite{ST,T}, and for the mass-critical case, see \cite{DT, DT2}. Moreover, in the mass-subcritical case, Serra and Tentarelli \cite{ST1}, using genus theory, proved the existence of multiple bound states with negative energy.

In the mass-supercritical case ($p>6$), the energy functional is no longer bounded below on the mass constraint. Moreover, scaling techniques (commonly used in the analysis on $\mathbb{R}^N$) are no longer applicable, and the Pohoz\v{a}ev identity is not available on metric graphs. This creates additional difficulties when proving the boundedness of Palais-Smale sequences. Furthermore, the topological properties of the metric graphs add another layer of complexity. To address these challenges, several works have focused on the mass-supercritical case. In \cite{Bort0}, Borthwick, Chang, Jeanjean, and Soave proved the existence of bounded Palais-Smale sequences carrying Morse index type information for functionals exhibiting a minimax geometry under the mass constraint. In \cite{Cha}, Chang, Jeanjean, and Soave proved the existence of normalized non-constant bound states for \eqref{eqs} on compact metric graphs when the prescribed mass is sufficiently small, using abstract critical point theory, Morse index type information, and blow-up analysis. Subsequently, in \cite{Bort}, for arbitrarily prescribed mass, Borthwick, Chang, Jeanjean, and Soave investigated the existence of normalized solutions for \eqref{eqsloc} on noncompact metric graphs. In \cite{Do0}, Dovetta, Jeanjean, and Serra studied the existence of normalized solutions for \eqref{eqs} on periodic metric graphs, as well as on noncompact metric graphs with finitely many edges under suitable topological assumptions. 

\subsection{Nonlinear magnetic Schr\"odinger equations on metric graphs}
In several important applications, in particular those related to nanotubes, the magnetic Schr\"odinger operator on metric graphs arises in a natural way (see \cite{Har,SDD}). These applications, together with the theoretical interest in the operator, have motivated its study and the investigation of its properties in various settings (see, for instance, \cite{KS}). From a mathematical point of view, previous research has mainly focused on spectral theory (see \cite{BGP,IK,Kor,KK2,Pan1,Pan2} and references therein). To the best of our knowledge, only a limited literature has been devoted to existence results for nonlinear equations involving this operator, especially from a variational point of view.

\subsubsection{The functional setting}
Inspired by \cite{Hof}, where a first attempt in our direction was made, we consider weaker smoothness assumptions on the magnetic potential $A$ and weaker integrability conditions on both $A$ and the electric potential $V$. We first develop the theory of the {\em magnetic Sobolev space} (and prove the self-adjointness of the associated {\em magnetic Schr\"odinger operator}) on locally finite connected metric graphs $\G=(\mathrm{V},\mathrm{E})$, in the sense of the following definition.
\begin{definition}\label{deflf}
    A connected metric graph $\G=(\mathrm{V},\mathrm{E})$ is {\em locally finite} if $\mathrm{E}$ is an at most countable set such that any bounded subset of the graph intersects at most finitely many edges; that is, for any bounded subset $B \subset \G$, the set $\{e\in\E: e \cap B\neq \emptyset\}$ is finite, where we regard edges $e \in \mathrm{E}$ as subsets of $\G$. 
\end{definition}
\noindent It is clear that any connected metric graph with a finite number of edges is locally finite. Moreover, if a connected metric graph $\G=(\mathrm{V},\mathrm{E})$ with at most countably many edges satisfies $\inf\limits_{e\in \mathrm{E}}\ell_e>0$ and $\operatorname{deg}(v)<+\infty$ for every $ v \in \mathrm{V}$, where $\operatorname{deg}(v)$ is the total number of edges incident at the vertex $v$, then $\G$ is locally finite. For more details on locally finite connected metric graphs, see \cite{Hof}.

These results are of independent interest and provide a solid foundation for the investigation of nonlinear magnetic Schr\"odinger equations on metric graphs.

Since a function $u: \G \to \mathbb{C}$ can be identified with a family of functions  $u_e: I_e \to \mathbb{C}$, $e\in\E$, the usual $L^p$ space on $\G$ is defined by
$$
L^p(\mathcal{G}, \mathbb{C}):=\bigoplus_{e \in \mathrm{E}} L^p(I_e, \mathbb{C}),
$$
endowed with the norm
$$
\norm{u}_p := \left(\sum_{e \in \mathrm{E}}\norm{u_e}_{L^p(I_e, \mathbb{C})}^p\right)^\frac{1}{p}, \quad \text { if } p \in[1, +\infty), \quad \text { and } \quad \norm{u}_{\infty}:=\sup_{e \in \mathrm{E}}\norm{u_e}_{L^{\infty}(I_e, \mathbb{C})}.
$$
The inner product in $L^2(\mathcal{G}, \mathbb{C})$ is
\begin{equation*}
(u,v)_2=\operatorname{Re} \int_\G u\bar{v}\,dx=\sum_e\operatorname{Re}\int_{I_e} u_e\bar{v}_e\,dx,
\end{equation*}
where $\operatorname{Re}(w)$ denotes the real part of $w \in \C$ and $\bar{w}$ is its conjugate. 

The Sobolev space $H^1(\G,\C)$ is defined by
$$
H^1(\G,\C):=\left\{u\in C(\G,\mathbb{C}): u' \in L^2(\G,\C) \text{ and } u\in L^2(\G,\C)\right\},
$$
and consists of continuous functions $u : \G \to \mathbb{C}$ such that $u_e \in  H^1([0,\ell_e],\C)$ for every edge $e\in \E$ (note that $\ell_e$ may be $+\infty$). Its norm is defined as
\begin{equation*}
    \norm{u}_{H^1(\G,\C)} = \left[\underset{e \in \E}\sum \left(\norm{u'}^2_{L^2(I_e,\C)}+\norm{u}^2_{L^2(I_e,\C)}\right)\right]^\frac{1}{2}.
\end{equation*}

Assume that $A \in L_{\rm loc}^2(\G,\R)$ and $V \in L_{\rm loc}^1(\G,\R)$, i.e., for every bounded subset $B\subset\G$ we have $A \in L^2(B, \R)$ and  $V \in L^1(B, \R)$, and that $V \geq 1$ a.e. on $\G$.
We define the  magnetic Sobolev space $H^1_A(\G,\C)$ by 
$$
H^1_A(\G,\C):=\left\{u\in C(\G,\mathbb{C}): D_Au \in L^2(\G,\C) \text{ and } \int_\G V(x)\abs{u}^2\,dx= \sum_{e\in \E}\int_{I_e} V_e(x)\abs{u_e}^2\,dx < +\infty\right\},
$$
where $D_A:=\displaystyle \frac{1}{i} \frac{\mathrm{d}}{\mathrm{d} x} - A(x)$. Its norm is
\begin{equation*}
    \norm{u} :=\left( \int_\G \abs{D_Au}^2+  V(x)\abs{u}^2\,dx\right)^\frac{1}{2},
\end{equation*}
and the inner product is
\begin{equation*}
(u,v):=\operatorname{Re}\int_\G \left(D_Au\overline{D_Av}+  V(x)u\bar{v}\right)\,dx.
\end{equation*}
We will show in Lemma \ref{lemhilbert} that $H^1_A(\G,\C)$ is a separable Hilbert space. Accordingly, we regard $D^2_A + V(x)$ as the unique self-adjoint operator associated with the
quadratic form 
$$\int_\G \left(D_Au\overline{D_Av}+  V(x)u\bar{v}\right)\,dx, \quad u,v \in H^1_A(\G,\C)$$ (see Section \ref{secsa}).

\subsubsection{Our results}

We will study the existence and multiplicity of normalized solutions for the following  nonlinear magnetic Schr\"odinger equations
\begin{equation}
\label{eqmagcompact}
\begin{cases}
   \displaystyle \left( \frac{1}{i} \frac{\mathrm{d}}{\mathrm{d} x} - A_e(x) \right)^2 u_e + V_e(x)u_e - \lambda u_e =  \abs{u_e}^{p-2} u_e, \quad &\forall e \in \mathrm{E}, \vspace{5pt}\\
\displaystyle\underset{e \succ v}\sum \left(\frac{1}{i}u'_e(v) - A^\pm_e(v)u_e(v)\right)= 0, &  \forall v \in \mathrm{V},
\end{cases}
\end{equation}
on compact metric graphs $\mathcal{G} = (\mathrm{V}, \mathrm{E})$ and
\begin{equation}
\label{eqmagloc}
\begin{cases}
   \displaystyle\left( \frac{1}{i} \frac{\mathrm{d}}{\mathrm{d} x} - A_e(x) \right)^2 u_e + V_e(x)u_e - \lambda u_e = \chi_{\K} \abs{u_e}^{p-2} u_e, \quad &\forall e \in \mathrm{E},\vspace{5pt}\\
\displaystyle\underset{e \succ v}\sum (\frac{1}{i}u'_e(v) - A^\pm_e(v)u_e(v))= 0, &  \forall v \in \mathrm{V},
\end{cases}
\end{equation}
on noncompact metric graphs $\mathcal{G} = (\mathrm{V}, \mathrm{E})$ with a finite number of edges and a non-empty compact core $\K$.\\
Here $p > 2$ and 
$A: \G \to \mathbb{R}$, $V: \G \to [1, +\infty)$ are families of functions $A_e: I_e \to \mathbb{R}$ and $V_e: I_e \to [1,+\infty)$.\\
Moreover, if $e\in\E$ and $v\in\V$, the notation $e \succ v$ means that the edge $e$ is incident at $v$, and $u'_e(v)$ (standing for $u'_e(0)$ or $-u'_e(\ell_e)$) is always the outward derivative, depending on whether the vertex $v$ is identified with $0$ or $\ell_e$. The magnetic potential $A$ is a $1d$ vector field (it changes sign when the orientation of the edge is reversed), and $A^\pm_e(v)$ stands for $A_e(0)$ or $-A_e(\ell_e)$, depending on whether the vertex $v$ is identified with $0$ or $\ell_e$.\\
The conditions in the second line of \eqref{eqmagcompact} and \eqref{eqmagloc}, known as {\em Kirchhoff magnetic boundary conditions}, were introduced in \cite{Pa} (see also \cite{Har,RS,SDD} for details and physical models).

Due to the variational structure of \eqref{eqmagcompact} and \eqref{eqmagloc}, we consider the energy functional $E( \cdot,\G): H^1_A(\G,\C) \to \mathbb{R}$ defined by
\begin{equation*}%
    E(u,\G) := \frac{1}{2}\norm{u}^{2}-\Psi(u,\G),
\end{equation*}
where
$$\Psi(u,\G):=
\begin{cases}
     \displaystyle\frac{1}{p}\int_\G  \abs{u}^p\,dx &\quad \text{if $\G $ is compact},\vspace{5pt}\\
     \displaystyle \frac{1}{p}\int_\K  \abs{u}^p\,dx &\quad \text{if $\G$ is noncompact with a non-empty compact core $\K$},
\end{cases}
$$
and we seek solutions as critical points of $E( \cdot,\G)$.

In both cases considered above, standard symmetric variational arguments yield infinitely many critical points of the unconstrained functional $E(\cdot,\G)$ whose energy levels tend to $+\infty$ (see, e.g., \cite[Section 3.3]{W}).

Next, we investigate the existence and multiplicity of normalized solutions to \eqref{eqmagcompact} and \eqref{eqmagloc} under the mass constraint
$$\int_\G\abs{u}^2\,dx=\mu, \quad \mu>0,$$
and thus we study the existence and multiplicity  of critical points of the functional $E( \cdot,\G)$ constrained on the $L^2$-sphere
\begin{equation*}%
    H_\mu(\G)= \left\{ u \in H^1_A(\G,\C):\int_\G \abs{u}^2\,dx = \mu \right\}.
\end{equation*}

To establish some of our existence results, we consider the following assumption.
\begin{itemize}
    \item[(\namedlabel{AssG}{G})] $D^2_A + V(x)$ admits a ground state that does not vanish identically on $\K$, i.e.,
    \begin{equation}\label{l1}
\lambda_1:=\underset{u \in H^1_A(\G,\C)\backslash\{0\}}{\inf}\frac{\norm{u}^2}{\norm{u}^2_2}
\end{equation}
     is achieved by $\varphi_1$ and ${\varphi_1}_{\left.\right|_{\K}} \not \equiv 0$.
\end{itemize}

Moreover, since $D^2_A + V(x)$ is self-adjoint and $V \geq 1$, we have $\sigma(D^2_A + V(x)) \subset [1, +\infty)$. Thus we will also consider the following assumption.
\begin{itemize}
    \item[(\namedlabel{AssGm}{G$_m$})] There exist $m\geq 1$ (possibly $+\infty$) distinct
eigenvalues below $\inf\sigma_{\rm ess}(D^2_A + V(x))$ (if $\sigma_{\rm ess}(D^2_A + V(x))=\emptyset$, then we set $\inf\sigma_{\rm ess}(D^2_A + V(x))=+\infty$), i.e., $D^2_A + V(x)$ has an increasing sequence of $m$ eigenvalues $1\leq \lambda_1<  \lambda_2< \ldots<\lambda_j<\ldots$ such that, for every $j$, $\lambda_j <\inf\sigma_{\rm ess}(D^2_A + V(x))$. 
\end{itemize}

By the definition of essential spectrum  (see, e.g., \cite{Reed,Cheverry,DBort}), if $\lambda <\inf\sigma_{\rm ess}(D^2_A + V(x))$ is an eigenvalue, then it has finite multiplicity. 

These assumptions describe the existence of linear bound states. In our arguments, both \eqref{AssG} and \eqref{AssGm} are used to establish compactness, while \eqref{AssGm} additionally yields multiplicity results. In the nonmagnetic setting, spectral assumptions ensuring \eqref{AssGm} with $m=1$, which implies \eqref{AssG}, have also been used in the study of normalized ground states in \cite{CFN17,Cac18}.
Spectral assumptions analogous to \eqref{AssGm} with $m>1$ have also been used to study nonlinear bound states and their dynamics in Euclidean spaces (see, e.g.,  \cite{TY02,CM21,GW08}).

Our main results on the existence and multiplicity of normalized solutions to \eqref{eqmagcompact} and \eqref{eqmagloc} are as follows.
\begin{theorem}
\label{th2}
    Let $\G$ be any compact metric graph or any noncompact metric graph with a finite number of edges and a non-empty compact core $\K$, and let $p > 2$. Assume that \eqref{AssG} or \eqref{AssGm} holds. Then, for any $\mu > 0$, one of the following two alternatives occurs:
    \begin{enumerate}[label=\rm(\roman*),ref=\roman*]
        \item \label{th2case1} either $E(\cdot, \G)$ has a critical point $u \in H^1_A(\G,\C)$ constrained on $H_\mu(\G)$ with Lagrange multiplier $\lambda \in (-\infty, \lambda_1]$ (and $\lambda \in (-\infty, \lambda_1)$ if \eqref{AssG} holds);
\item \label{th2case2} or, for every $\lambda \in (-\infty, \lambda_1)$, $E(\cdot, \G)$ has a critical point $u \in H^1_A(\G,\C)$ constrained on $H_\nu(\G)$ for some $0<\nu < \mu$ and the corresponding Lagrange multiplier is $\lambda$.
    \end{enumerate}
\end{theorem}

\begin{theorem}\label{th3}
    Under the assumptions of Theorem \ref{th2}, there exists $\mu_0>0$ such that for any $0<\mu \leq \mu_0$, the first alternative of Theorem \ref{th2} occurs: $E(\cdot, \G)$ has a critical point $u \in H^1_A(\G,\C)$ constrained on $H_\mu(\G)$ with Lagrange multiplier $\lambda \in [0, \lambda_1]$ (and $\lambda \in [0, \lambda_1)$ if \eqref{AssG} holds). Moreover, if $2<p<6$, then for every $\mu >0$, $E(\cdot, \G)$ has a critical point $u \in H^1_A(\G,\C)$ constrained on $H_\mu(\G)$ with Lagrange multiplier $\lambda \in (-\infty, \lambda_1]$ (and $\lambda \in (-\infty, \lambda_1)$ if \eqref{AssG} holds).
\end{theorem}

By applying \cite[Theorem 3.5]{W}, we obtain the following multiplicity results.
\begin{theorem}\label{th4}
 Let $\G$ be any compact metric graph or any noncompact metric graph with a finite number of edges and a non-empty compact core $\K$, and let $p > 2$. Assume that \eqref{AssGm} holds. Then, for any $j =1,2,\ldots,m$ (if $m=+\infty$, then, for any $j \in \mathbb{N}^+$), there exists $\mu_{j} > 0$ depending on $\G$, $p$, $A$, and $V$ such that, for any $0 < \mu < \mu_{j}$, $E( \cdot,\G)$ has at least $j$  distinct critical points $u_1, u_2,\ldots,u_j$ constrained on $H_\mu(\G)$. Moreover, if $2<p\leq 6$, $m=+\infty$, and $\lim\limits_{j\to+\infty}\lambda_j=+\infty$, then, for any $\mu >0$ ($0<\mu \leq \mu_0$ if $p=6$), $E( \cdot,\G)$ has infinitely many distinct critical points constrained on $H_\mu(\G)$ whose energy levels tend to $+\infty$.
\end{theorem}

We observe that, for instance, if $A$ is continuously differentiable on every edge $e\in \E$ and $V$ is continuous on every edge $e\in \E$, if $u\in H_\mu(\G)$ is a critical point of $E( \cdot,\G)$ constrained on $H_\mu(\G)$, then $u$ is twice continuously differentiable on every edge $e\in \E$ and there exists a Lagrange multiplier $\lambda \in \R$ such that $u$ is a solution of \eqref{eqmagcompact} or \eqref{eqmagloc}, respectively (see Appendix \ref{regularity} for details).

Theorem \ref{th4}  extends the multiplicity results for nonlinear Schr\"odinger equations on metric graphs in \cite{Do,Jeanjean} to nonlinear magnetic Schr\"odinger equations on compact metric graphs for all $p>2$.  

On compact metric graphs, in  \cite{Do},  Dovetta proved the existence of infinitely many normalized solutions to \eqref{eqs} in the mass-subcritical case, for every value of the mass, and in the mass-critical case for masses below a certain threshold (where normalized ground states exist if and only if the mass is less than or equal to this threshold). The author also characterized the relation between this threshold and the topology of the graph. In the mass-supercritical case, in \cite{Jeanjean},
by an analysis of constant-sign solutions combined with a new type of linking  and gradient flow techniques on the mass constraint,  Jeanjean and Song proved the multiplicity of sign-changing normalized solutions to \eqref{eqs} when the prescribed mass is sufficiently small. The multiplicity arguments in \cite{Do,Jeanjean} rely on the compactness of the embedding $H^1(\G) \hookrightarrow L^2(\G)$.

In this paper, we prove the multiplicity of normalized solutions to \eqref{eqmagcompact} and \eqref{eqmagloc} without using the compact embedding of $H^1_A(\G,\C)$ into $L^2(\G,\C)$. Moreover, due to the presence of a magnetic potential, the problem must be studied in complex-valued function spaces; therefore, the argument in \cite{Jeanjean} based on the signs of solutions is not directly applicable to \eqref{eqmagcompact} and \eqref{eqmagloc}. In addition, due to the presence of potentials, the relation between the range of masses for which the equations admit a normalized solution and the topology of the graph changes, and the study of the signs of normalized solutions becomes meaningless.

Finally, let $\G=(\V,\E)$ be any noncompact metric graph with a finite number of edges (whether its compact core is empty or non-empty).  Under assumption \eqref{AssGm}, we study the existence and multiplicity of normalized solutions for 
\begin{equation}
\label{eqmagnoncompact}
\begin{cases}
   \displaystyle \left( \frac{1}{i} \frac{\mathrm{d}}{\mathrm{d} x} - A_e(x) \right)^2 u_e + V_e(x)u_e - \lambda u_e =  \abs{u_e}^{p-2} u_e, \quad &\forall e \in \mathrm{E}, \vspace{5pt}\\
\displaystyle\underset{e \succ v}\sum \left(\frac{1}{i}u'_e(v) - A^\pm_e(v)u_e(v)\right)= 0, &  \forall v \in \mathrm{V}.
\end{cases}
\end{equation}

Using the same notation as above for the energy functional with
$$\Psi(u,\G):=\frac{1}{p}\int_\G  \abs{u}^p\,dx,$$
we obtain the following results.
\begin{theorem}\label{thn3}
    Let $\G$ be any noncompact metric graph with a finite number of edges, and let $p > 2$. Assume that \eqref{AssGm} holds. Then there exists $\mu_{0}^*>0$ depending on $\G$, $p$, $A$, and $V$ such that, for any $0 < \mu \leq \mu_{0}^*$, $E(\cdot, \G)$ has a critical point $u \in H^1_A(\G,\C)$ constrained on $H_\mu(\G)$,  with Lagrange multiplier $\lambda \in [0, \lambda_1)$.
\end{theorem}

\begin{theorem}\label{thn4}
Let $\G$ be any noncompact metric graph with a finite number of edges, and let $p > 2$. Assume that \eqref{AssGm} holds. Then, for any $k=1,2,\ldots, m$, there exists $\mu_{k}^* > 0$ depending on $\G$, $p$, $A$, and $V$ such that, for any $0 < \mu < \mu_{k}^*$, $E(\cdot, \G)$ has at least $k$ distinct critical points $u_1, u_2,\ldots,u_k$ constrained on $H_\mu(\G)$. 
 \end{theorem}

Theorem \ref{thn4} extends the multiplicity results in \cite{Do,Jeanjean} to nonlinear magnetic Schr\"odinger equations on noncompact metric graphs under spectral assumptions for all $p>2$. For this class of graphs, the main difficulty is the lack of compactness of the embedding $H^1_A(\G,\C)\hookrightarrow L^p(\G,\C)$.

\subsubsection{Further comments and remarks}

\begin{remark}
For any compact metric graph $\G$, since $\K=\G$, by Remark \ref{recompact} below, the assumption \eqref{AssG} is satisfied.
\end{remark}

\begin{remark}\label{regrow}
    If $\G$ is a compact metric graph, or if $\G$ is a noncompact metric graph with a finite number of edges and $V$ also satisfies
    \begin{equation}\label{eq:Vgrowth}
    \lim_{\operatorname{dist}(x,x_0)\to +\infty}V(x)=+\infty \text{ for some } x_0\in \G,
    \end{equation}
    then, by Remark \ref{recompact} and, for noncompact graphs, arguing as in     \cite[Section 3]{BW}, we obtain that $H^1_A(\G,\C)$ is compactly embedded in $L^2(\G,\mathbb{C})$.
    Thus, 
    by \cite[Theorem 4.22]{DBort}, Theorem \ref{thminmiax}, and Remark \ref{remThm27},
    $$m =+\infty\text{ in }\eqref{AssGm}, \quad
    \lim\limits_{j\to+\infty}\lambda_j = +\infty, \quad
    \inf\sigma_{\rm ess}(D^2_A + V(x)) = +\infty,
    $$
    so that \eqref{AssGm} is satisfied and, therefore, when the compact core $\K$ of $\G$ is non-empty, Theorems \ref{th2}--\ref{th4} hold for $E(\cdot, \G)$ with
    $$\Psi(u,\G):=\frac{1}{p}\int_\K \abs{u}^p\,dx.$$
    Moreover, if $\G$ is noncompact and $V$ satisfies condition \eqref{eq:Vgrowth}, then, whether $\G$ has a non-empty compact core or not, since the embedding $H^1_A(\G,\C) \hookrightarrow L^p(\G,\mathbb{C})$ is also compact (by Gagliardo-Nirenberg-Sobolev inequality, see Lemma \ref{lemgnsg} below), Theorems \ref{th2}--\ref{th4} under assumption \eqref{AssGm} also hold for
$E(\cdot,\G)$ with 
$$\Psi(u,\G):=\frac{1}{p}\int_\G \abs{u}^p\,dx.$$
Observe that Theorems \ref{thn3} and \ref{thn4} only require assumption \eqref{AssGm}, without the compact embedding considered here. Under condition \eqref{eq:Vgrowth}, however, the conclusion of Theorem \ref{th4} yields infinitely many distinct normalized critical points for every $\mu>0$ if $2<p<6$, and for $0<\mu\leq\mu_0$ if $p=6$, whose energy levels tend to $+\infty$.
\end{remark}

\begin{remark}\label{resubgraph}
    For any metric graph $\G$ with a finite number of edges and a non-empty compact subgraph $\K^*$, Theorems \ref{th2} (with $\K$ in assumption \eqref{AssG} replaced by $\K^*$), \ref{th3}, and \ref{th4} hold for
$$
E_{\K^*}( \cdot,\G): H^1_A(\G,\C) \to \mathbb{R},
\quad
E_{\K^*}(u,\G) := \frac{1}{2}\norm{u}^{2}- \frac{1}{p}\int_{\K^*}  \abs{u}^p\,dx.
$$ 
\end{remark}

\begin{remark}\label{recedge}
Theorems \ref{thn3} and \ref{thn4} and Remarks \ref{regrow} and \ref{resubgraph} hold on any noncompact connected metric graph that has at most countably many edges and satisfies $\inf\limits_{e\in \mathrm{E}}\ell_e>0$ and $\operatorname{deg}(v)<+\infty$ for every $v \in \mathrm{V}$.
Indeed, these geometric assumptions imply that $\G$ is locally finite. Thus, the results of Section \ref{secsa}, including the local compactness stated in Remark \ref{recompact}, remain applicable. Moreover, the diamagnetic inequality and the Gagliardo-Nirenberg-Sobolev inequality (see \cite[Section 2.3]{Do0}) yield the inequalities stated in Lemma \ref{lemgnsg} for noncompact graphs. 
The compactness arguments in Section \ref{seccompact}, adapted to the whole-graph nonlinearity in Section \ref{secnoncompact} for Theorems \ref{thn3} and \ref{thn4}, then apply under assumption \eqref{AssGm}.

For Remark \ref{regrow}, the same argument under condition \eqref{eq:Vgrowth} gives the compact embedding $H^1_A(\G,\C)\hookrightarrow L^p(\G,\C)$ for every $2\leq p<+\infty$. For localized nonlinearities, we consider a fixed non-empty compact subgraph $\K^*$, as in Remark \ref{resubgraph}, and use the compact embedding $H^1_A(\G,\C)\hookrightarrow L^p(\K^*,\C)$. Thus, Remarks \ref{regrow} and \ref{resubgraph} also hold under their respective assumptions.
\end{remark}

\begin{remark}
    For any $V_q \in L^q(\G,\R)$, where $1\leq q \leq +\infty$, and $\nu>0$ sufficiently large, Theorems \ref{th2}--\ref{thn4} and Remarks \ref{regrow}--\ref{recedge} extend to the magnetic Schr\"odinger operator $D_A^2 +\left(V(x) +V_q(x)+\nu\right)$, provided that $D_A^2 +\left(V(x) +V_q(x)+\nu\right)$ satisfies the corresponding spectral assumption \eqref{AssG} or \eqref{AssGm} (see Section \ref{subsecineq}).
\end{remark}

\begin{remark}\label{refu}
    In view of \cite{Alves}, when  \eqref{AssGm} holds, ($\sigma_{\rm ess}(D^2_A + V(x))$ in \eqref{AssGm} could be replaced by $\sigma_{\rm ess}(D^2_A + V(x))-K_2$, where $K_2$ is defined in \cite[Section 1]{Alves}) and the nonlinearity $\abs{u}^{p-2}u$ (or $\chi_{\K}\abs{u}^{p-2}u$) could be replaced by $f(|u|^2)u$ (or $\chi_{\K}f(|u|^2)u$, respectively) where $f$ satisfies the assumptions in \cite[Section 6.2]{Alves} with $2^*=+\infty$.
    All the results above (with respect to assumption \eqref{AssGm}) then hold with minor modifications, additionally assuming, if necessary to preserve the strict upper bounds on the Lagrange multipliers (see e.g. Theorem \ref{thn3}), that the associated nonlinear potential is strictly positive away from zero.
\end{remark}

\begin{remark}
     For nonlinear magnetic Schr\"odinger equations, or for nonlinear Schr\"odinger equations with general potentials on noncompact graphs, arguments involving Morse index type information as in \cite{Bort} fail due to the presence of potentials (in particular, due to the scaling arguments on half-lines used in \cite[Lemma 3.2]{Bort}).
\end{remark}

\subsubsection{Plan of the paper}
The rest of the paper is organized as follows. In Section \ref{secsa}, we develop the theory of magnetic Sobolev spaces on metric graphs and establish the self-adjointness of the magnetic Schr\"odinger operator. In Section \ref{sec:preliminaries}, we prove a Gagliardo-Nirenberg-Sobolev inequality, introduce some modified functionals, and derive a nonexistence result that plays a crucial role in our approach.  In Section \ref{seccompact}, we prove the compactness of Palais-Smale sequences with energy below $\mu \inf\sigma_{\rm ess}(D^2_A + V(x))/2$. Section \ref{secp23} is devoted to the proofs of Theorems \ref{th2} and \ref{th3}, while Section \ref{secp4} contains the proof of Theorem \ref{th4}. Finally, in Section \ref{secnoncompact}, we prove Theorems \ref{thn3} and \ref{thn4}.

In the following, we denote by $C$ a positive constant that may change from inequality to inequality.

\section{Magnetic Sobolev space and self-adjointness of the magnetic Schr\"odinger operator}\label{secsa}  
In this section, we consider a locally finite connected metric graph $\G=(\mathrm{V},\mathrm{E})$ (see Definition \ref{deflf})
assuming that $A: \G \to \mathbb{R}$ and $V: \G \to [1, +\infty)$ satisfy $A \in L_{\rm loc}^2(\G,\R)$ and $V \in L_{\rm loc}^1(\G,\R)$. 

Since here we are interested in general properties of the {\em magnetic Schr\"odinger operator}
\[
D_A^2+V(x)=\left(\frac{1}{i} \frac{\mathrm{d}}{\mathrm{d} x} - A(x)\right)^2 + V(x),
\]
in this section only, we regard
\[
    L^2(\G,\mathbb{C})
    :=
    \left\{u=(u_e)_{e \in \mathrm{E}}:u_e \in L^2(I_e,\mathbb{C})\text{ for every }e\in \mathrm{E} \text{ and }\sum_{e\in \mathrm{E}}\norm{u_e}^2_{L^2(I_e,\mathbb{C})}=\int_\G\abs{u}^2\,dx<+\infty\right\},
\]
$$
H^1(\G,\C):=\left\{u\in C(\G,\mathbb{C}): u' \in L^2(\G,\C) \text{ and } u\in L^2(\G,\C)\right\},
$$
and
$$H^1_A(\G,\C)
:=\left\{u\in C(\G,\mathbb{C}) : D_Au \in L^2(\G, \mathbb{C}) \text{ and } \int_\G V(x)\abs{u}^2\,dx< +\infty\right\}
$$
as complex Hilbert spaces.\footnote{Observe that here $\G$ is locally finite and so, with respect to the analogous definitions given in the introduction, we need, as additional condition, to require the finiteness of the sums.}
Of course, in such a case, we have to consider on them different inner products\footnote{By an abuse of notation we will denote the {\em new} inner products in the same way as before.} and so, as usual (see e.g. \cite[Section 11.4]{Ha}), we equip them  with 
$$
(u,v)_2 := \int_\G u\bar{v}\,dx, \quad u,v \in L^2(\G,\mathbb{C}),
$$
and
\begin{equation*}%
    (u,v) :=\int_\G \left(D_Au\overline{D_Av}+  V(x)u\bar{v}\right)\,dx, \quad u,v \in H^1_A(\G,\C)
\end{equation*}
respectively, which actually define the same norms as before.

Considering each edge separately, for every $e\in \mathrm{E}$, let us define
$$H^1_{A_e}(I_e,\C)
:=\left\{u_e\in C(I_e,\mathbb{C}) : D_{A_e}u_e \in L^2(I_e, \mathbb{C}) \text{ and } \int_{I_e} V(x)\abs{u_e}^2\,dx< +\infty\right\},
$$
where
$$D_{A_e}:=\frac{1}{i} \frac{\mathrm{d}}{\mathrm{d} x} - A_e(x),$$
equipped with the inner product
\begin{equation*}
(u_e,v_e)_{H^1_{A_e}(I_e,\C)}:=\int_{I_e} \left(D_{A_e}u_e\overline{D_{A_e}v_e}+  V_e(x)u_e\overline{v}_e\right)\,dx, \quad u_e,v_e \in H^1_{A_e}(I_e,\C).
\end{equation*}
The norm induced by this inner product is given by
$$
\norm{u_e}_{H^1_{A_e}(I_e,\C)}:=\left[\int_{I_e} \abs{D_{A_e}u_e}^2+  V_e(x)\abs{u_e}^2\,dx\right]^\frac{1}{2}.
$$
Observe that, if $u\in H^1_A(\G,\C)$, then, for every $e\in{\rm E}$, $u_e\in H^1_{A_e}(I_e,\C)$.

For $H^1_A(\G,\C)$  we have the following fundamental result.
\begin{lemma}\label{lemhilbert}
Assume that $A: \G \to \mathbb{R}$ and $V: \G \to [1, +\infty)$ satisfy $A \in L_{\rm loc}^2(\G,\R)$ and $V \in L_{\rm loc}^1(\G,\R)$. Then $H^1_A(\G,\C)$ is a separable Hilbert space.
\end{lemma}
\begin{proof}
First observe that, for any $u \in H^1_A(\G,\C)$ and $e \in \E$, since $A_e \in L_{\rm loc}^2(I_e,\R)$  and $u_e \in L^2(I_e,\mathbb{C})$, we know that $u_e'=iD_{A_e}u_e + iA_e(x)u_e \in L^1_{\rm loc}(I_e, \mathbb{C})$, so that $u_e \in W^{1,1}_{\rm loc}(I_e, \mathbb{C})$.

Now, let us show that, for every $e \in \mathrm{E}$, $H^1_{A_e}(I_e,\C)$ is a Hilbert space.
Let $\{u_n\}$ be a Cauchy sequence in $H^1_{A_e}(I_e,\C)$. Then, $\{u_n\}$ is bounded in $H^1_{A_e}(I_e,\C)$ and $\{u_n\}$ and $\{D_{A_e}(u_n
)\}$ are Cauchy sequences in $L^2(I_e,\mathbb{C})$. It follows that $\{u_n\}$ converges to some limit $u$ and $\{D_{A_e}(u_n
)\}$ converges to some limit $g$ in $L^2(I_e,\mathbb{C})$. Hence, by Fatou's Lemma, we know that 
$$\int_{I_e}V_e(x)\abs{u}^2\,dx
\leq \liminf_{n\to +\infty} \int_{I_e}V_e(x)\abs{u_n}^2\,dx
<+\infty.$$
Moreover, since  for all $\varphi\in C_c^\infty(I_e,\mathbb{C})$, $A_e(x)\varphi \in L^2(I_e,\mathbb{C})$, we have

\[
\int_{I_e}u_n\overline{\varphi}'\,dx
= - \int_{I_e}u_n'\overline{\varphi}\,dx
=-i\int_{I_e}\left(D_{A_e}u_n\;\overline{\varphi} +A_e(x)u_n\overline{\varphi}\right)\,dx,
\]
taking the limit as $n\to +\infty$, we get
$$
\int_{I_e}u\overline{\varphi}'\,dx 
= -i\int_{I_e}\left(g\overline{\varphi} + A_e(x)u \overline{\varphi}\right)\,dx.
$$
Then, by the definition of weak derivatives, we obtain $u'=ig+iA_e(x)u \in L_{\rm loc}^1(I_e,\mathbb{C})$. %
Thus, by \cite[Theorem 8.2]{Ha} $u$ is continuous on $I_e$, $g=D_{A_e}u$, and $u \in H^1_{A_e}(I_e,\C)$. Moreover, Fatou's lemma and the Cauchy property give
$$
\|u_n-u\|_{H^1_{A_e}(I_e,\C)}^2
\leq\liminf_{m\to+\infty}\|u_n-u_m\|_{H^1_{A_e}(I_e,\C)}^2
\longrightarrow0\quad\text{as }n\to+\infty.
$$
Thus $H^1_{A_e}(I_e,\C)$ is complete and hence is a Hilbert space.

Now, as in \cite[Proposition 8.1]{Ha}, we show that $H^1_{A_e}(I_e,\C)$ is separable. It is clear that the operator $B: H^1_{A_e}(I_e,\C) \to  L^2(I_e,\mathbb{C})\times L^2(I_e,\mathbb{C})$ defined by $B u = (\sqrt{V_e(x)}u, D_{A_e}u)$ is an isometry.
 Since $L^2(I_e,\mathbb{C})\times L^2(I_e,\mathbb{C})$ is separable, then $B(H^1_{A_e}(I_e,\C))$ is separable as well and so we can conclude.

Finally, define the space
$$\underset{e \in \mathrm{E}}{\bigoplus}H^1_{A_e}(I_e,\C):=\left\{u=(u_e)_{e \in \mathrm{E}}:u_e \in H^1_{A_e}(I_e,\C)\text{ for every }e\in \mathrm{E} \text{ and }\sum_{e\in \mathrm{E}}\norm{u_e}^2_{H^1_{A_e}(I_e,\C)}<+\infty \right\}$$ 
equipped with the inner product
$$
(u,v)_{\underset{e \in \mathrm{E}}{\bigoplus} H^1_{A_e}(I_e,\C)}:=\sum_{e\in \mathrm{E}} (u_e,v_e)_{H^1_{A_e}(I_e,\C)}.
$$
The norm induced by this inner product is given by
$$
\norm{u}_{\underset{e \in \mathrm{E}}{\bigoplus} H^1_{A_e}(I_e,\C)}:=\left(\sum_{e\in \mathrm{E}}\norm{u_e}^2_{H^1_{A_e}(I_e,\C)}\right)^\frac{1}{2}=\left[\int_{\G} \abs{D_{A}u}^2+  V(x)\abs{u}^2\,dx\right]^\frac{1}{2}.
$$

Since $\E$ is at most countable and each $H^1_{A_e}(I_e,\C)$ is a separable Hilbert space, their Hilbert direct sum $\underset{e \in \mathrm{E}}{\bigoplus}H^1_{A_e}(I_e,\C)$ is also a separable Hilbert space (see \cite[Proposition 5.5.2]{Conway}).

Finally, to conclude by applying \cite[Theorem 1.22]{Adams}, we show that $H^1_A(\G,\C)$ is a closed subspace of $\underset{e \in \mathrm{E}}{\bigoplus} H^1_{A_e}(I_e,\C)$.
First observe that for any edge $e \in \E$, for any $u_e \in  H^1_{A_e}(I_e,\C)$, and any bounded interval $I \subset I_e$, we know that 
$$\begin{aligned}
    \norm{u_e}_{W^{1,1}(I,\mathbb{C})} &= \norm{u_e'}_{L^1(I,\C)}+\norm{u_e}_{L^1(I,\C)}\\ &\leq \norm{D_{A_e}u_e}_{L^1(I,\C)}+ \norm{A_e(x)u_e}_{L^1(I,\C)}+\norm{u_e}_{L^1(I,\C)}\\
    &\leq C\norm{D_{A_e}u_e}_{L^2(I,\C)}+ \norm{A_e}_{L^2(I,\C)}\norm{u_e}_{L^2(I,\C)}+C\norm{u_e}_{L^2(I,\C)},
\end{aligned}
$$ 
where $C$ depends on $I$. Thus, $ H^1_{A_e}(I_e,\C)$ is continuously embedded in $W^{1,1}(I, \C)$ and so, by \cite[Theorems 8.2 and 8.8]{Ha}, we conclude that $ H^1_{A_e}(I_e,\C)$ is continuously embedded in $C(\overline{I},\C)$.
Thus, if $\{u_n\} \subset {\underset{e \in \mathrm{E}}{\bigoplus} H^1_{A_e}(I_e,\C)}$ and $u_n \to u$ in $\underset{e \in \mathrm{E}}{\bigoplus}H^1_{A_e}(I_e,\C)$, we have
$$u_e(v)=\lim_{n \to +\infty}u_{n,e}(v)$$
for any $v\in\V$ and $e\succ v$.
Now, let $\{u_n\} \subset H^1_A(\G,\C)$ and $u_n \to u$ in $\underset{e \in \mathrm{E}}{\bigoplus}H^1_{A_e}(I_e,\C)$. Then, for any $v\in\V$ and $e,f\succ v$,
\[
u_e(v)
= \lim_{n \to +\infty} u_{n,e}(v)
=\lim_{n \to +\infty} u_{n,f}(v)
= u_f(v)
\]
and so $u \in H^1_A(\G,\C)$ concluding the proof.
\end{proof}

\begin{remark}\label{recompact}
In the proof of Lemma \ref{lemhilbert}, we have shown that for any edge $e \in \E$ and any bounded interval $I \subset I_e$, we have that $H^1_{A_e}(I_e,\C)$ is continuously embedded in $W^{1,1}(I, \C)$. Since $\G$ is locally finite and $W^{1,1}(I, \C)$ is compactly embedded in $L^p(I, \C)$ for all $1\leq p<+\infty$ (see \cite[Theorem 8.8]{Ha}), we conclude that for any bounded set $B \subset \G$, $H^1_A(\G,\C)$ is compactly embedded in $L^p(B, \C)$ for all $1\leq p<+\infty$.
Moreover, arguing as in \cite[Section 8.7]{Lieb}, we know that, if $u_n\rightharpoonup u$ in $H^1_A(\G,\C)$, then,
up to a subsequence, $u_n \to u$ almost everywhere on $\G$.
\end{remark}

Now,  define the 
densely defined quadratic form \(Q\) on \(L^2(\mathcal G,\mathbb C)\times L^2(\G,\C)\) such that
$$
Q(u,v)=(u,v), \quad\forall u,v \in \operatorname{dom}(Q):=H^1_A(\G,\C).
$$
Denoting
$$C_c^\infty(\mathcal{G},\mathbb{C}):=\left\{u=(u_e)_{e \in \mathrm{E}}:u_e \in C_c^\infty(I_e,\mathbb{C})\text{ for every }e\in \mathrm{E} \text{ and }\operatorname{supp}u \text{ is bounded}\, \right\}$$
and following \cite{Serov}, we have
\begin{theorem}\label{thext}
Assume that $A: \G \to \mathbb{R}$ and $V: \G \to [1, +\infty)$ satisfy $A \in L_{\rm loc}^2(\G,\R)$ and $V \in L_{\rm loc}^1(\G,\R)$. Then there exists a unique self-adjoint operator $T$ with domain $\operatorname{dom}(T)\subset H^1_A(\G,\C)$ such that
$$
Q(u,v)=(Tu,v)_2,\quad \text{for every } u \in \operatorname{dom}(T),v \in H^1_A(\G,\C).
$$
\end{theorem}
\begin{proof}
By \cite[Theorem 29.2]{Serov} it is enough to show that, starting from the Hilbert space $L^2(\G,\mathbb{C})$, the quadratic form $Q$ is a densely defined, closed, semibounded, and symmetric quadratic form (see \cite[Definition 29.1]{Serov}).\\
It is clear that $C_c^\infty(\mathcal{G},\mathbb{C}) \subset H^1_A(\G,\C)\subset L^2(\G,\mathbb{C})$. Let $u \in L^2(\G,\mathbb{C})$. Then, for any $\varepsilon>0$, there exists a finite subset $\E_\varepsilon$ of $\E$ such that
\[
\|u\|_{L^2{(\E_\epsilon^c,\mathbb{C})}}^2:=\sum_{e\in \E_\epsilon^c} \|u_e\|_{L^2{(I_e,\mathbb{C})}}^2 < \frac{\varepsilon^2}{4}.
\]
Then,
$$
u_{\varepsilon}
:=(u_{\varepsilon,e}), \text{ where } u_{\varepsilon,e}
:=\begin{cases}
    u_e\quad  \text{if }e\in \E_\varepsilon,\\
    0\quad \text{otherwise,}
\end{cases}
$$
satisfies $\norm{u-u_\varepsilon}_2<\varepsilon/2$.
Moreover, since by 
\cite[Corollary 4.23]{Ha},  $C_c^\infty(I_e,\C)$ is dense in $L^2(I_e,\mathbb{C})$, for every edge $e \in \E$, we know that there is $\tilde{u}_{\varepsilon} \in C_c^\infty(\mathcal{G},\mathbb{C})$ such that $\norm{\tilde{u}_{\varepsilon}-u_\varepsilon}_2<\varepsilon/2$. Hence, $\norm{\tilde{u}_{\varepsilon}-u}_2<\varepsilon$ and so we obtain that $C_c^\infty(\mathcal{G},\mathbb{C})$ is dense in $L^2(\G,\mathbb{C})$, namely that the quadratic form $Q$ is densely defined.\\
By the definition of $Q$, it is clear that the quadratic form $Q$ is symmetric.\\
Since $Q(u,u) \geq 0$ for all $u \in \operatorname{dom}(Q)$, the quadratic form $Q$ is semibounded from below.\\
Finally, since $V \geq 1$, we have  $Q(u,u) \geq \norm{u}^2_2$ for every $u \in \operatorname{dom}(Q)$. Then,
\begin{equation*}%
\norm{u}_Q := \sqrt{Q(u,u)+\norm{u}_2^2}, \quad  u \in \operatorname{dom}(Q),
\end{equation*}
is an equivalent norm on $H^1_A(\G,\C)$.
Thus, from Lemma \ref{lemhilbert} it follows that  the quadratic form $Q$ is closed.
\end{proof}

Let now
\[
H^1_{A,c}(\G,\C):=\left\{u\in H^1_A(\G,\C): \operatorname{supp}u \text{ is bounded}\,\right\}.
\]
Moreover, for $A$ and $V$ continuous on every edge $e \in \E$, let us consider
$$
\widetilde{H}^2_c(\G,\mathbb{C})
:=\left\{
u=(u_e)_{e \in \mathrm{E}}\in C(\G,\mathbb{C}) :
\begin{array}{c}
\operatorname{supp}u \text{ is bounded},
u'' \in L^2(\G,\mathbb{C}),
u_e'\in C(I_e,\mathbb{C}) \ \forall e \in \mathrm{E},\\
\displaystyle\underset{e \succ v}\sum \Big(\frac{1}{i}u'_e(v) - A_e^\pm(v)u_e(v)\Big)= 0, \quad  \forall v \in \mathrm{V}
\end{array}
\right\}.
$$

Using the following density result, we can clarify the relationship between $T$ and $D_A^2 + V(x)$. The proof of the next lemma follows arguments from \cite{Hof, Lieb}.

\begin{lemma}\label{lemdense}
Assume that $A: \G \to \mathbb{R}$ and $V: \G \to [1, +\infty)$ satisfy $A \in L_{\rm loc}^2(\G,\R)$ and $V \in L_{\rm loc}^1(\G,\R)$. Then $H^1_{A,c}(\G,\C)$ is a dense subset of $H^1_A(\G,\C)$. Moreover, for $A$ and $V$ continuous on every edge $e \in \E$, $\widetilde{H}^2_c(\G,\mathbb{C})$ is a dense subset of $H^1_A(\G,\C)$.
\end{lemma}
\begin{proof}
    We first show that $H^1_{A,c}(\G,\C)$ is a dense subset of $H^1_A(\G,\C)$. Fix some $x_0 \in \G$.
    Take $\chi \in C_c^\infty([0,+\infty),\R)$, with $0 \leq  \chi \leq 1$ and $\chi \equiv 1$ in $[0,1]$, and consider $\chi_n(r)=\chi(r/ n)$ for $r \in [0,+\infty)$ and $n\in\mathbb{N}^+$.
    Define the cut-off functions $\psi_n: \G \to \R$ via
    $$
    \psi_n(x):=\chi_n\left(\operatorname{dist}(x_0,x)\right).
    $$
    Then, for every $e \in \E$, $\psi_{n,e}:=\psi_n|_{I_e} \in W^{1,1}(I_e,\mathbb{R})$. For any $u \in H^1_A(\G,\C)$, we have that $\operatorname{supp}(\psi_{n}u)$ is bounded. Moreover, arguing as in the proof of Lemma \ref{lemhilbert}, for every $e \in \E$, $u_e\in W^{1,1}(I_e\cap \operatorname{supp}(\psi_n),\mathbb{C})$ and so, by 
    \cite[Corollary 8.10]{Ha}, we have that
    $$ 
    D_{A}(\psi_{n}u) =\psi_{n} D_{A}u -i\psi_{n}'u
    $$
    so that 
    \[
    \|D_{A}(\psi_{n}u)\|_2
    \leq \|D_{A}u\|_2 + C(n) \|u\|_2<+\infty
    \]
    and
    $$
    \norm{D_A (u-\psi_nu )}_2 
    \leq \norm{(1-\psi_n)D_A u}_2 
    + \frac{1}{n}\sup_{r \in [0,+\infty)}\abs{\chi'(r)}\norm{u}_2.
    $$
    Thus, $\psi_nu\in H^1_{A,c}(\G,\C)$ and by the Lebesgue dominated convergence theorem it follows that $\psi_{n}u \to u$ in $H^1_A(\G,\C)$.

    Next, we show that $\widetilde{H}^2_c(\G,\mathbb{C})$ is a dense subset of $H^1_A(\G,\C)$  for $A,V$  continuous on every edge $e \in \mathrm{E}$. Fix some $x_0 \in \V$.
    For $R > 0$, set $B_R:=\{x\in \G:\operatorname{dist}(x,x_0)<R\}$. Since for every $z_1,z_2\in\mathbb{C}$,
    \[
    \frac{1}{2}|z_1|^2
    \leq|z_1-z_2|^2+ |z_2|^2
    \leq 2|z_1|^2+3|z_2|^2,
    \]
    then, for any $R>0$  and $u \in H^1_A(\G,\C)$ or $u \in H^1(\G,\C)$, we have
\begin{equation}\label{eqbr}
\frac{1}{2} \int_{B_R}\abs{u'}^2\,dx
\leq \int_{B_R} \left(\left|\left(\frac{1}{i} \frac{\mathrm{d}}{\mathrm{d} x} - A(x)\right)u\right|^2+ \abs{A(x)}^2\abs{u}^2\right)\,dx
\leq \int_{B_R}\left(2\abs{u'}^2+ 3\abs{A(x)}^2\abs{u}^2\right)\,dx.
\end{equation}
Hence, since $\G$ is locally finite and $A,V$ are continuous on every edge $e \in \mathrm{E}$, we deduce that, for any $u: \G \to \mathbb{C}$ such that $\operatorname{supp} u \subset B_R$, $u \in H^1_A(\G,\C)$ if and only if $u \in H^1(\G,\C)$. Thus, $\widetilde{H}^2_c(\G,\mathbb{C}) \subset H^1_A(\G,\C)$.

Fix $u \in H_{A,c}^1(\G,\mathbb{C})$ and fix some $R_0>0$ such that $\operatorname{supp} u \subset B_{R_0}$.
Then, for any $R_1 > R_0$, by \eqref{eqbr}, we know that, if there exists a sequence $\{u_n\}\subset H^1(\G,\C)$ such that $u_n \to u$ in $H^1(\G,\C)$ and $\operatorname{supp} u_n \subset B_{R_1}$ for all $n \geq 1$, then $u_n \to u$ in $H^1_{A}(\G,\C)$. 

Let $R_1>R_0$ be such that, for any bounded edge $e$ that intersects $B_{R_0}$, we have $e \subset B_{R_1}$ (the existence of such an $R_1$ is ensured by the local finiteness of $\G$) and let us
construct a sequence $\{u_n\}\subset \widetilde{H}^2_c(\G,\mathbb{C})$ such that $u_n \to u$ in $H^1(\G,\C)$ and $\operatorname{supp} u_n \subset B_{R_1}$ for all $n \geq 1$.\\
For every bounded edge $e$ intersecting $B_{R_0}$, choose $\varphi_{n,e}\in C^\infty(I_e,\C)$ converging to $u_e$ in $H^1(I_e,\C)$ (see \cite[Section 5.3.3]{evans}). For every half-line $e$ intersecting $B_{R_0}$, choose $R_e^*>0$ such that $\operatorname{supp}u_e\subset[0,R_e^*)\subset B_{R_0}$. By \cite[Theorem 8.6 and Remark 14 in Section 8]{Ha}, the even extension of $u_e$ belongs to $H_0^1((-R_e^*,R_e^*),\C)$. Approximating this extension by functions in $C_c^\infty((-R_e^*,R_e^*),\C)$ and then restricting them to $[0,R_e^*)$ and extending them by zero on $[R_e^*,+\infty)$, we obtain $\varphi_{n,e}\in C^\infty(I_e,\C)$ with $\operatorname{supp}\varphi_{n,e}\subset B_{R_0}$ and $\varphi_{n,e}\to u_e$ in $H^1(I_e,\C)$. If $u_e=0$, we may simply take $\varphi_{n,e}=0$.

Let now
$$
\varphi_n: \G \to \mathbb{C},
\quad
\varphi_{n,e}:=
\begin{cases}
   \varphi_{n,e} &\quad\text{if } e\cap B_{R_0}\neq \emptyset, \\
    0 &\quad\text{if } e\cap B_{R_0}= \emptyset
\end{cases}
\text{for every }e\in\E.
$$
Then, $\operatorname{supp} \varphi_{n}\subset B_{R_1}$ for all $n \geq 1$ and $\varphi_{n,e} \to u_e$ in $H^1(I_e,\C)$ for every $e \in \E$.

However, the functions $\varphi_n$ may not belong to $H^1(\G,\C)$ since they may not be continuous at some vertices $v \in \V$. Nevertheless, we use $\{\varphi_n\}$ to construct $\{u_n\}\subset \tilde{H}_c^2(\G,\C)$ such that $u_n \to u$ in $H^1(\G,\C)$ as follows.
By \cite[Theorems 8.2 and 8.8]{Ha}, we know that $H^1(I_e, \C) \hookrightarrow C(I_e,\mathbb{C})$ for every $e\in\E$. Then, since $\G$ is locally finite (so that $\{e \in \E: e\cap B_{R_0}\}$ is a finite set), we can find
a subsequence $\{\varphi_{\nu_n}\}$ such that for all $e \in \{e \in \E: e\cap B_{R_0}\neq\emptyset\}$,
\begin{equation}\label{equ-phin}
\norm{u_e-\varphi_{\nu_n,e}}_{C(I_e,\mathbb{C})}
=\norm{u_e-\varphi_{\nu_n,e}}_{L^\infty(I_e,\C)}
<\frac{1}{2^n} \quad\text{for all } n \geq 1,
\end{equation}
For any $v \in \V$ and $e \succ v$, using the notation introduced for $A_e^\pm(v)$, define
\begin{align*}
c_{n,0,v,e}&:=u_e(v)-\varphi_{\nu_n,e}(v),\\
c_{n,1,v,e}&:=iA_e^\pm(v)u_{e}(v)-\varphi_{\nu_n,e}'(v) +2n c_{n,0,v,e},\\
\widetilde{c}_{n,1,v,e}&:= \max\left\{n,\abs{c_{n,1,v,e}}^3\right\}.
\end{align*}
Moreover, for every incidence of an edge $e$ at a vertex $v$, let $a$ be the corresponding endpoint of $I_e$ and set $x_{e,v}(x):=|x-a|$, $x\in I_e$.\footnote{All endpoint traces, outgoing derivatives, and coefficients, including those defined above, are understood separately for each incidence.}
Let us consider\footnote{For a loop, the corresponding edge component of each correction function is understood as the sum of the two endpoint contributions. The norm estimates below apply to each individual endpoint contribution.}
$$\psi_{n,0,v}: \G \to \C,
\quad
{\psi_{n,0,v,e}}(x)
:=\begin{cases}
    c_{n,0,v,e}\left(1-n x_{e,v}\right)^2&\quad\text{if }  \ e \succ v, \ x_{e,v}\leq 1/n, \\
    0 &\quad\text{otherwise},
\end{cases}
$$
and
$$
\psi_{n,1,v}: \G \to \C,
\quad
{\psi_{n,1,v,e}}(x)
:=\begin{cases}
    c_{n,1,v,e}x_{e,v}(1-\widetilde{c}_{n,1,v,e}x_{e,v})^2&\quad\text{if } \ e \succ v, \ x_{e,v}\leq 1/\widetilde{c}_{n,1,v,e}, \\
    0 &\quad\text{otherwise}.
\end{cases}
$$
Observe that, for $n$ large enough, $\operatorname{supp}\psi_{n,0,v}, \operatorname{supp}\psi_{n,1,v}\subset B_{R_1}$.
Indeed, for any $v \not\in B_{R_1}$, since $\operatorname{supp}\varphi_{\nu_n}, \operatorname{supp}u \subset B_{R_1}$, we have that, for all $e\in\E$ with $e \succ v$, $c_{n,0,v,e}=c_{n,1,v,e}=0$ and so $\psi_{n,0,v}=\psi_{n,1,v}=0$. On the other hand, for any $v \in B_{R_1}$ and any $e \succ v$, since
\[
x_{e,v}\leq \frac{1}{\tilde{c}_{n,1,v,e}}=\frac{1}{\max\{n,|c_{n,1,v,e}|^3\}}\leq\frac{1}{n},
\]
it is enough to take $n$ large enough such that $\{x\in e : x_{e,v}\leq 1/n \} \subset B_{R_1}$.
\\
Moreover, by \eqref{equ-phin}, we have
$$\norm{\psi_{n,0,v,e}}_2^2 = \frac{\abs{c_{n,0,v,e}}^2}{5n} < \frac{C}{2^{2n}n},
\qquad
\norm{\psi_{n,0,v,e}'}_2^2=\frac{4n\abs{c_{n,0,v,e}}^2}{3}<C\frac{n}{2^{2n}},$$
and
\[
\norm{\psi_{n,1,v,e}}_2^2
\leq
\frac{\abs{c_{n,1,v,e}}^2}{5\widetilde{c}_{n,1,v,e}^3}\leq \frac{C}{n^\frac{7}{3}},
\qquad
\norm{\psi_{n,1,v,e}'}_2^2
\leq \frac{46\abs{c_{n,1,v,e}}^2}{15\widetilde{c}_{n,1,v,e}}
\leq \frac{C}{n^\frac{1}{3}}.
\]
Then, let
$$u_n := \varphi_{\nu_n} + \sum_{v\in B_{R_1}\cap \V}(\psi_{n,0,v} + \psi_{n,1,v}).$$
For $n \geq 1$ large enough, namely such that
\begin{equation*}%
\frac{1}{n} < \min_{ v\in B_{R_1}, e\succ v} \ell_e
\end{equation*}
and $\{x\in e : x_v\leq 1/n \} \subset B_{R_1}$ for any $v \in B_{R_1}$ and $e \succ v$, we have $u_n \in \widetilde{H}^2_c(\G,\mathbb{C})$. Indeed:
\begin{itemize}
\item $\operatorname{supp}u_n \subset B_{R_1}$ since, as observed above, the functions involved in the definition of $u_n$ have support contained in $B_{R_1}$;
\item $u_n\in C(\G,\mathbb{C})$ since, for any $v\in\V$ and $e,f\succ v$, if $v\notin B_{R_1}$ then $u_{n,e}(v)=u_{n,f}(v)=0$, if $v\in B_{R_1}$ then
\begin{align*}
u_{n,e}(v)
&=
\varphi_{\nu_n,e}(v) + \sum_{w\in \V \cap B_{R_1}}(\psi_{n,0,w,e}(v) + \psi_{n,1,w,e}(v))\\
&=
\varphi_{\nu_n,e}(v) + \psi_{n,0,v,e}(v) + \psi_{n,1,v,e}(v)
=
\varphi_{\nu_n,e}(v)+c_{n,0,v,e}
=
u_e(v)
\end{align*}
and analogously $u_{n,f}(v)=u_f(v)$, so that, being $u_e(v)=u_f(v)$, we get $u_{n,e}(v)=u_{n,f}(v)$;
\item obviously $u''_n \in L^2(\G,\mathbb{C})$ and $u_{n,e}'\in C(I_e,\mathbb{C})$ for every $e \in \mathrm{E}$;
\item if $v\in\V\cap B_{R_1}^c$,
\[
\sum_{e \succ v} \Big(\frac{1}{i}u'_{n,e}(v) - A_e^\pm(v)u_{n,e}(v)\Big)=0
\]
and if $v\in\V\cap B_{R_1}$,
\begin{align*}
\displaystyle\sum_{e \succ v} \Big(\frac{1}{i}u'_{n,e}(v) - A_e^\pm(v)u_{n,e}(v)\Big)
&=\frac{1}{i} \sum_{e \succ v} \Big(\varphi'_{\nu_n,e}(v) + \psi_{n,0,v,e}'(v)+\psi_{n,1,v,e}'(v)\Big)\\
&\quad
- \sum_{e \succ v} A_e^\pm(v) \Big(\varphi_{\nu_n,e}(v) + \psi_{n,0,v,e}(v)+\psi_{n,1,v,e}(v)\Big)\\
&=\frac{1}{i} \sum_{e \succ v} \Big(\varphi'_{\nu_n,e}(v)+ c_{n,1,v,e}-2nc_{n,0,v,e}\Big)\\
&\quad - \sum_{e \succ v} A_e^\pm(v)\Big(\varphi_{\nu_n,e}(v) +c_{n,0,v,e}\Big)\\
&=
\sum_{e \succ v}\Big(A_e^\pm(v)u_{e}(v) - A_e^\pm(v)u_{e}(v)\Big)
= 0.
\end{align*}
\end{itemize}
Therefore, as $n \to + \infty$,
$$
\norm{u-u_n}_{H^1(\G,\C)}\leq \sum_{e \in \E,\, e\cap B_{R_0}\neq \emptyset}\left(\norm{u_e - \varphi_{\nu_n,e}}_{H^1(I_e,\C)} + \sum_{\substack{v\in\V\cap B_{R_1}\\e\succ v}}\norm{\psi_{n,0,v,e} + \psi_{n,1,v,e}}_{H^1(I_e,\C)}\right) \to 0
$$
and so, by previous arguments, $u_n \to u$ in $H^1_{A}(\G,\C)$.

Finally, since $H^1_{A,c}(\G,\C)$ is a dense subset of $H^1_A(\G,\C)$, the density of $\widetilde{H}^2_c(\G,\C)$ in $H^1_A(\G,\C)$ follows.
\end{proof}

Now, let $T$ be the unique self-adjoint operator associated with $Q$ introduced in Theorem \ref{thext} and, assuming that $A$ is continuously differentiable and $V$ is continuous on every edge $e\in\E$, let $\operatorname{dom}\left(D_A^2 + V(x)\right):=\widetilde{H}^2_c(\G,\mathbb{C})$. Define the associated quadratic form of $D_A^2 + V(x)$ by
$$
Q_0(u,w):=((D_A^2 + V(x))u,w)_2, \quad u,w\in \operatorname{dom}(Q_0):= \operatorname{dom}\left(D_A^2 + V(x)\right).
$$
Recalling that the Friedrichs extension is a self-adjoint extension of a non-negative densely defined symmetric operator, we have
\begin{lemma}\label{lemfri}
    Let $\G$ be any locally finite connected metric graph. Let  $A: \G \to \mathbb{R}$ be continuously differentiable on every edge $e \in \mathrm{E}$ and $V: \G \to [1, +\infty)$ be continuous  on every edge $e \in \mathrm{E}$. Then the operator $T$ is the Friedrichs extension of $D_A^2 + V(x)$.  
\end{lemma}
\begin{proof}
    Since $C_c^\infty(\mathcal{G},\mathbb{C}) \subset \widetilde{H}^2_c(\G,\mathbb{C})\subset L^2(\G,\mathbb{C})$ and $C_c^\infty(\mathcal{G},\mathbb{C})$ is dense in $L^2(\G,\mathbb{C})$, it is clear that $D_A^2 + V(x)$ is densely defined in $L^2(\G,\C)$.
    We first show that $D_A^2 + V(x)$ is symmetric and nonnegative.
    Let $u,w \in \widetilde{H}^2_c(\G,\mathbb{C})$ and, for some fixed 
    $x_0 \in \G$, let $R>0$ be large enough such that $B_R:=\{x\in \G:\operatorname{dist}(x,x_0)<R\}$ satisfies $(\operatorname{supp}u)\cup(\operatorname{supp}w) \subset B_R$.
    Since $\G$ is locally finite, integrating by parts we have
\begin{equation*}
\begin{split}
\left(D_A^2u,w\right)_2-\left(D_Au,D_Aw\right)_2
&=\sum_{e\in\E} \int_{I_e} 
\left(-u_e'' 
+ i A'_eu_e
+ 2i A_e u'_e
+A_e^2 u_e\right)\overline{w_e}\,dx\\
&\quad-\sum_{e\in\E} \int_{I_e} \left(
u_e' \overline{w_e'}
-i A_eu_e\overline{w_e'}
+i A_e u'_e \overline{w_e}
+A_e^2 u_e\overline{w_e}\right)\,dx\\
&
=i\sum_{e\in\E_b}\left[
\left(\frac{1}{i}u_e'(0)-A_e(0)u_e(0)\right) \overline{w_e(0)}
-   \left(\frac{1}{i}u_e'(\ell_e)-A_e(\ell_e)u_e(\ell_e)\right)\overline{w_e(\ell_e)}\right]\\
&\qquad
+i\sum_{e\in\E\setminus\E_b}\left(\frac{1}{i}u_e'(0)-A_e(0)u_e(0)\right) \overline{w_e(0)}\\
&=i\sum_{e\in\E}\sum_{v \prec e}\left(\frac{1}{i}u'_e(v) -A_e^\pm(v)u(v)\right) \overline{w(v)}\\
&=
i\sum_{v\in\V}\left[\overline{w(v)}\sum_{e \succ v}\left(\frac{1}{i}u'_e(v) -A_e^\pm(v)u(v)\right)\right]
=0
\end{split}
\end{equation*}
where $\E_b:=\{e\in\E: e \text{ bounded}\}$, and, analogously,
\[
\left(u,D_A^2w\right)_2-\left(D_Au,D_Aw\right)_2
=0.
\]
Thus, $D_A^2 + V(x)$ is symmetric and nonnegative and, for every $u,w\in\operatorname{dom}\left(D_A^2 + V(x)\right):=\widetilde{H}^2_c(\G,\mathbb{C})=:\operatorname{dom}(Q_0)$,
\begin{equation}\label{Q0Q}
Q_0(u,w)=\left((D_A^2 + V(x))u,w\right)_2=\left(u,(D_A^2 + V(x))w\right)_2=\left(D_Au,D_Aw\right)_2 + (V(x)u,w)_2=Q(u,w).
\end{equation}
By Lemmas \ref{lemhilbert} and \ref{lemdense}, $\operatorname{dom}\left(D_A^2 + V(x)\right)$%
is dense in the Hilbert space $H^1_A(\G,\C)$  with respect to the inner product $(\cdot,\cdot)$.
Then,
by \cite[Theorem 29.4]{Serov}, we have that there exists a self-adjoint extension (Friedrichs extension) $\tilde{T}$ of $D_A^2 + V(x)$.
Moreover, by \eqref{Q0Q}, we have that for every $u,v\in H^1_A(\G,\C)=\operatorname{dom}(Q)$,
$$
Q(u,v)=\lim_{n \to +\infty}Q_0(u_n,v_n),
$$
where $u_n,v_n \in \operatorname{dom}(Q_0)$ and $u_n \to u$, $v_n \to v$ in $H^1_A(\G,\C)$.
Then, since by Theorem \ref{thext}, $T$ is the unique self-adjoint operator associated with $Q$ (recall that, in the proof of \cite[Theorem 29.4]{Serov}, $\tilde{T}$ is defined as the self-adjoint operator associated with $Q$), we conclude that $T=\tilde{T}$.
\end{proof}
\begin{remark}
Due to Lemma \ref{lemfri}, throughout this paper, unless otherwise specified, we denote the {\em Friedrichs extension} $T$ in Theorem \ref{thext} by $D_A^2 + V(x)$.
\end{remark}

We conclude this section recalling that, due to the assumptions on $A$ and $V$ and to the previous results, we can consider for $T$ the following classical result (see e.g. \cite[Theorem 5.15]{DBort}) to {\em justify} our assumptions (\ref{AssG}) and (\ref{AssGm}).
\begin{theorem}\label{thminmiax}
For $j \in \mathbb{N}^+$, let
$$
\widehat\lambda_j(T) := \inf_{M \in \mathcal{M}_j}\sup_{\substack{u \in M \setminus \{0\}}}\frac{(Tu,u)_2}{(u,u)_2},
$$ where $\mathcal{M}_j$ denotes the family of $j$-dimensional subspaces of $\operatorname{dom}(T)$. 
Then, for each $j \in \mathbb{N}^+$, one of the following alternatives holds.
\begin{enumerate}[label=\rm(\roman*)]
\item $\widehat\lambda_j(T)$ is the $j$-th eigenvalue (arranged in increasing order and counted with multiplicity) and there are at least $j$ eigenvalues below the essential spectrum.
\item $\widehat\lambda_j(T)=\inf \sigma_{\rm ess}(T)$ and there are at most $j-1$ eigenvalues below the essential spectrum.
\end{enumerate}
\end{theorem}
\begin{remark}\label{remThm27}
Throughout the paper, since the eigenvalues in Theorem \ref{thminmiax} need not be simple, we use $\lambda_1<\lambda_2<\ldots$ in \eqref{AssGm} to denote the distinct eigenvalues, counted without multiplicity.
\end{remark}

\section{Preliminaries}
\label{sec:preliminaries}
In this section we introduce some tools that will be useful in the remaining part of the paper and prove a nonexistence result.
\subsection{Diamagnetic inequality} 
We first deduce the diamagnetic inequality on locally finite connected metric graphs by following the argument of \cite[Theorem 7.21]{Lieb}. 
\begin{lemma}\label{lemdiam}
Let $\G$ be a locally finite connected metric graph, $A: \G \to \mathbb{R}$ and $V: \G \to [1, +\infty)$ satisfying  $A \in L_{\rm loc}^2(\G,\R)$ and $V \in L_{\rm loc}^1(\G,\R)$, and $u \in H^1_A(\G,\C)$. Then $\abs{u} \in H^1(\G,\mathbb{R})$ and the {\em diamagnetic inequality}
    \begin{equation}\label{eqdiam}
            \abs{\abs{u}'(x)} \leq \abs{D_Au(x)}
    \end{equation}
    holds pointwise for almost every $x \in\G$.
\end{lemma}
\begin{proof}
For any $u \in H^1_A(\G,\C)$ and any $e \in \E$, as in \cite[Theorem 6.17]{Lieb}, we have
    \begin{equation}\label{eqabsu'}
        \abs{u_e}'(x) = 
        \begin{cases}
            \displaystyle \operatorname{Re}\left(\frac{\overline{u_e}}{\abs{u_e}}u_e'\right)(x)\quad &\text{if } u_e(x)\neq0,\\
            0\quad &\text{if } u_e(x)=0.
        \end{cases}
    \end{equation}
Since, if $u_e(x)\neq 0$,
$$\operatorname{Re}\left(\frac{\overline{u_e}}{\abs{u_e}}iA_eu_e\right)(x)=\operatorname{Re}\left(iA_e\abs{u_e}\right)(x)=0,$$
     \eqref{eqabsu'} reads as
        \begin{equation*}
        \abs{u_e}'(x) = 
        \begin{cases}
            \displaystyle 
\operatorname{Re}\left(\frac{i\overline{u_e}}{\abs{u_e}}D_{A_e}u_e\right)(x)\quad &\text{if } u_e(x)\neq0,\\
            0\quad &\text{if } u_e(x)=0.
        \end{cases}
    \end{equation*}
    Then \eqref{eqdiam} follows from the fact that $\abs{\operatorname{Re}z}\leq\abs{z} $.\\
    Finally, since $D_Au \in L^2(\G,\C)$, we conclude that $\abs{u}' \in L^2(\G,\R)$. Moreover, by the definition of $H^1_A(\G,\C)$, we know that $\abs{u}$ is continuous on $\G$ and $\abs{u} \in L^2(\G,\R)$, and thus, $\abs{u} \in H^1(\G,\mathbb{R})$.
    \end{proof}
\subsection{Gagliardo-Nirenberg-Sobolev inequalities}
\label{subsecineq}
In the remainder of this paper, we always assume that $\G$ is a compact graph or noncompact metric graph with  finitely many edges. Thus, in this subsection, we first introduce the Gagliardo-Nirenberg-Sobolev inequalities on compact graphs and noncompact metric graphs with  finitely many edges. We then extend  our results from $V$ which satisfies $V: \G \to [1, +\infty)$  and $V \in L_{\rm loc}^1(\G,\R)$, to $V+V_q$ for any $V_q \in L^q(\G,\R)$ with $1\leq q\leq +\infty$.

The next lemma is an immediate consequence of Lemma \ref{lemdiam} and of the Gagliardo-Nirenberg-Sobolev inequality \cite[Theorem 1]{Ha2} for compact graphs (applying it to every edge $e \in \E$ and then summing up) and \cite[Section 2]{AST2} for noncompact metric graphs with finitely many edges.
\begin{lemma}\label{lemgnsg}
Let $p \geq 2$. Let $A: \G \to \mathbb{R}$ and $V: \G \to [1, +\infty)$ satisfy $A \in L_{\rm loc}^2(\G,\R)$ and $V \in L_{\rm loc}^1(\G,\R)$. For any compact metric graph $\G$, there exist $C_{p,\G}, C_{\infty,\G}>0$ such that, for every $u \in H^1_A(\G,\C)$,
\[
    \|u\|_p^p \leq C_{p,\G} \norm{u}^{\frac{p}{2}-1}\norm{u}_2^{\frac{p}{2}+1}
\text{ and }
    \norm{u}_\infty\leq C_{\infty,\G} \norm{u}^\frac{1}{2}\norm{u}_2^\frac{1}{2}.
\]
For any noncompact metric graph $\G$ with a finite number of edges, there exist
 $C_{p,\G}, C_{\infty,\G}>0$ such that, for every $u \in H^1_A(\G,\C)$,
\[
\|u\|_p^p \leq C_{p,\G}  \norm{D_Au}^{\frac{p}{2}-1}_2\norm{u}_2^{\frac{p}{2}+1}
\text{ and }
\norm{u}_\infty\leq C_{\infty,\G} \norm{D_Au}_2^\frac{1}{2} \norm{u}_2^\frac{1}{2}.
\]
\end{lemma}
Then, for all $V_q \in L^q(\G,\R)$ with 
$1\leq q < +\infty$, by Lemma \ref{lemgnsg}, H\"older's inequality and Young's inequality, we have that, for any $\varepsilon>0$, there exists $C_\varepsilon>0$ such that, for every $u \in H^1_A(\G,\C)$,
$$
\int_{\G} \abs{V_q(x)}\abs{u}^2\,dx
\leq \norm{V_q}_q\norm{u}^2_{\frac{2q}{q-1}}
\leq C \norm{V_q}_q\norm{u}^\frac{1}{q}\norm{u}_2^\frac{2q-1}{q}
\leq \varepsilon\norm{u}^2 +C_\varepsilon\norm{V_q}_q^{\frac{2q}{2q-1}}\norm{u}_2^2\,
\text{ for }q>1,
$$
\[
\int_{\G} \abs{V_q(x)}\abs{u}^2\,dx
\leq \norm{V_q}_1 \norm{u}^2_{\infty}
\leq C \norm{V_q}_q\norm{u}\norm{u}_2
\leq \varepsilon\norm{u}^2 +C_\varepsilon\norm{V_q}_q^2\norm{u}_2^2\,
\text{ for }q=1.
\]
Let $\varepsilon=1/2$. Then, for all $\displaystyle \nu\geq 1 + C_\frac{1}{2}\norm{V_q}_q^{\frac{2q}{2q-1}}$, we obtain that, for every $u \in H^1_A(\G,\C)$,
$$
\frac{1}{2}\norm{u}^2 + \norm{u}_2^2 \leq \int_\G \abs{D_Au}^2+  \left(V(x) +V_q(x)+\nu\right)\abs{u}^2\,dx \leq \frac{3}{2}\norm{u}^2 + (2\nu-1)  \norm{u}_2^2.
$$
Moreover, for $V_\infty \in L^\infty(\G,\R)$, we have that, for all $\nu\geq  \norm{V_\infty}_\infty$ and $u \in H^1_A(\G,\C)$,
$$
\norm{u}^2  \leq \int_\G \abs{D_Au}^2+  \left(V(x) +V_\infty(x)+\nu\right)\abs{u}^2\,dx \leq \norm{u}^2 + 2\nu  \norm{u}_2^2.
$$
Then, for any $V_q \in L^q(\G,\R)$ with $1\leq q \leq +\infty$, for suitable $\nu$'s, we can define an equivalent norm on $H^1_A(\G,\C)$ by
$$
\left(\int_\G \abs{D_Au}^2+  \left(V(x) +V_q(x)+\nu\right)\abs{u}^2\,dx\right)^\frac{1}{2}.
$$ 
Hence, Theorems \ref{th2}--\ref{th4}, \ref{thn3}, \ref{thn4} and Remarks \ref{regrow}--\ref{recedge} can be extended to the magnetic Schr\"odinger operator $D_A^2 +\left(V(x) +V_q(x)+\nu\right)$, provided that $D_A^2 +\left(V(x) +V_q(x)+\nu\right)$ satisfies the corresponding assumption \eqref{AssG} or \eqref{AssGm}. 

\subsection{A nonexistence result}\label{secnon}
Let us set
\begin{equation}\label{mucp}
\mu_{c,p}:=
\begin{cases}
    C_{p,\G}^{\frac{2}{2-p}}\lambda_1^{\frac{6-p}{2p-4}}  &\text{ if } 2<p\leq 6,\\
    \displaystyle C_{p,\G}^{\frac{2}{2-p}}\left(\frac{2pc}{p-2}\right)^{\frac{6-p}{2p-4}} &\text{ if } p > 6,
\end{cases}
\end{equation}
for any $c>0$, where $\lambda_1$ and $C_{p,\G}$ are defined respectively in \eqref{l1} and in Lemma \ref{lemgnsg}, and
$$\mu_{\lambda, p}^*:=
\left(\frac{4}{6-p}\right)^\frac{6-p}{2p-4}\left(\frac{4}{p-2}\right)^\frac{1}{2}C_{p,\G}^{\frac{2}{2-p}}\abs{\lambda}^{\frac{6-p}{2p-4}}
$$
for $\lambda<0$ and $2<p<6$.

Here we present a key lemma establishing a nonexistence result, which rules out case (\ref{th2case2}) in Theorem \ref{th2} and plays a crucial role in the proofs of Theorems \ref{th3}, \ref{th4}, \ref{thn3} and \ref{thn4}.
\begin{lemma}\label{lemnonex}
    For any $c>0$ and $0<\nu<\mu_{c,p}$, there exists no $u \in H_\nu(\G)$ such that $E'(u,\G) = 0$ and $E(u,\G)\leq c\mu_{c,p}$. Moreover, if $2<p<6$, then for any $\lambda<0$ and $0<\nu<\mu_{\lambda,p}^*$, there exists no $u \in H_\nu(\G)$ such that $\langle E'(u,\G),\cdot \rangle-\lambda(u,\cdot)_2 = 0$.
\end{lemma}
\begin{proof}
Let us prove the first part. Arguing by contradiction, assume that there exists $u \in H_\nu(\G)$ such that $E'(u,\G) = 0$ and $E(u,\G)\leq c\mu_{c,p}$ for some $0<\nu<\mu_{c,p}$.\\    
Let $2<p< 6$. Since $E'(u,\G) = 0$, by Lemma \ref{lemgnsg} we have
\begin{equation}\label{eqnonex}
\norm{u}^2
= p\Psi(u,\G) 
\leq   C_{p,\G}\nu^{\frac{p}{4}+\frac{1}{2}}\norm{u}^{\frac{p}{2}-1},
\end{equation}
and thus, since $\norm{u}^2 \geq \lambda_1 \norm{u}^2_2$ for all $u \in H^1_A(\G,\C)$,
$$
\lambda_1^{\frac{3}{2}-\frac{p}{4}}\nu^{\frac{3}{2}-\frac{p}{4}}\leq \norm{u}^{3-\frac{p}{2}} \leq C_{p,\G}\nu^{\frac{p}{4}+\frac{1}{2}},
$$
which leads to a contradiction, since $\nu < \mu_{c,p}=C_{p,\G}^{\frac{2}{2-p}}\lambda_1^{\frac{6-p}{2p-4}}$.\\ For $p=6$, \eqref{eqnonex} contradicts $\nu < \mu_{c,6}=C_{6,\G}^{-\frac{1}{2}}$.\\ 
Now, let $p>6$. Then, it follows from $E'(u,\G) = 0$ and $E(u,\G)\leq c\mu_{c,p}$ that
$$
\Psi(u,\G) = \frac{2E(u,\G) - \langle E'(u,\G),u\rangle}{p-2}\leq \frac{2c\mu_{c,p}}{p-2}.
$$
By $E'(u,\G) = 0$ and Lemma \ref{lemgnsg}, we obtain
$$
\norm{u}^2
= p\Psi(u,\G) 
= \left(p\Psi(u,\G)\right)^\frac{4}{p-2}\left(p\Psi(u,\G)\right)^\frac{p-6}{p-2} 
\leq  C_{p,\G}^\frac{4}{p-2}\norm{u}^2\nu^{\frac{p+2}{p-2}}\left(\frac{2pc\mu_{c,p}}{p-2}\right)^{\frac{p-6}{p-2}},
$$
and thus,
$$
C_{p,\G}^{-\frac{4}{p-2}}\left(\frac{2pc}{p-2}\right)^{\frac{6-p}{p-2}}\leq\nu^{\frac{p+2}{p-2}}\mu_{c,p}^{\frac{p-6}{p-2}} <\mu_{c,p}^{\frac{2p-4}{p-2}},
$$
which leads to a contradiction with the definition of $\mu_{c,p}$.

To prove the second part, let $2<p<6$ and, arguing by contradiction, assume that there exists $u \in H_\nu(\G)$ and $\lambda<0$ such that $\langle E'(u,\G),\cdot \rangle-\lambda(\cdot,u)_2 = 0$. Then, by  Lemma \ref{lemgnsg} and Young's inequality, we deduce that 
$$
\begin{aligned}
    \norm{u}^2 -\lambda\norm{u}_2^2 &= p\Psi(u,\G)
    \leq C_{p,\G}\nu^{\frac{p}{4}+\frac{1}{2}}\norm{u}^{\frac{p}{2}-1}
    = \left(\frac{4}{p-2}\right)^\frac{2-p}{4}C_{p,\G}\nu^{\frac{p+2}{4}}\left(\frac{4}{p-2}\right)^\frac{p-2}{4}\norm{u}^{\frac{p-2}{2}} \\
    &\leq \frac{6-p}{4}\left(\frac{4}{p-2}\right)^\frac{2-p}{6-p}C_{p,\G}^\frac{4}{6-p}\nu^\frac{p+2}{6-p}+\norm{u}^2,
\end{aligned}
$$
namely $\nu\geq \mu_{\lambda,p}^*$, which leads to a contradiction.
\end{proof}

\section{Modified functionals}\label{seccompact}

Let us fix $p>2$ and $\mu >0$. As in \cite{Alves,Es,Bu}, to find  critical points of $E(\cdot, \G)$ constrained on $H_\mu (\G)$, we first look for  critical points of the following family of modified functionals
$$
E_{r, \mu}(u,\G) := \frac{1}{2}\norm{u}^{2}-\Psi(u,\G)-H_{r,\mu}(u, \G),\; u \in U_{\mu},
$$
where $U_{\mu} := \{u \in H_A^1(\G,\C): \norm{u}_2^2<\mu\}$, $H_{r, \mu}(u)$ is a penalization term defined by
$$
H_{r, \mu}(u,\G) := f_{r}\left(\frac{\norm{u}_2^2}{\mu}\right),  \text { with } f_{r}(s) := \frac{s^{r}}{1-s} \; (0 \leq s<1),
$$
and $r>1$ is a parameter that will be chosen large enough.

A straightforward computation gives
\begin{equation}\label{eqfr's}
    f_{r}^{\prime}(s)=\frac{r s^{r-1}}{1-s}+\frac{s^{r}}{(1-s)^{2}}>\frac{r}{s} f_{r}(s)>0, \; \quad \text{for } s \in (0,1)
\end{equation}
and
\begin{equation} \label{eqfr''s}
	f''_r(s)=\frac{r(r-1)s^{r-2}}{1-s}+\frac{rs^{r-1}}{(1-s)^2}+\frac{rs^{r-1}}{(1-s)^2}+\frac{2s^{r}}{(1-s)^3}>0,\; \quad \text{for } s \in (0,1),
\end{equation}
and so $f'_r$ is increasing for $s \in [0,1)$.

Fixing
$$h_r(s)=f'_r(s)s-f_r(s) \text{ for } s \in [0,1),$$
the information above yields
\begin{equation} \label{eqhrs}
	h'_r(s)=f''_r(s)s>0, \; \quad \text{for } s \in (0,1),
\end{equation}
showing that $h_r$ is an increasing function for $s \in [0,1)$.

It is standard to prove
that $E_{r, \mu}(\cdot, \G) \in C^1(U_{\mu},\R)$ and for any $u \in U_{\mu}$ and $v \in H_A^1(\G,\C)$, one has, for any compact metric graph $\G$,
$$
\langle E_{r, \mu}'(u,\G),v\rangle = (u,v)-\operatorname{Re}\int_\G \abs{u}^{p-2}u\bar{v} \,dx-\frac{2}{\mu}f'_r\left(\frac{\norm{u}_2^2}{\mu}\right)(u,v)_2,
$$
or, for any noncompact metric graph $\G$ with a non-empty compact core $\K$,
$$
\langle E_{r, \mu}'(u,\G),v\rangle = (u,v)-\operatorname{Re}\int_\K \abs{u}^{p-2}u\bar{v} \,dx-\frac{2}{\mu}f'_r\left(\frac{\norm{u}_2^2}{\mu}\right)(u,v)_2.
$$

Next, we prove that, for sufficiently large $r>1$, Palais--Smale (PS) sequences of $E_{r,\mu}(\cdot,\G)$ converge, up to a subsequence, to critical points of $E_{r,\mu}(\cdot,\G)$. We then study the limit of these critical points as $r\to+\infty$. The limit is either a critical point of $E(\cdot,\G)$ constrained on $H_\mu(\G)$ or a nontrivial unconstrained critical point with mass strictly smaller than $\mu$.

Fix $\mu >0$ and assume that, for any $r >1$ sufficiently large, there exists a sequence $\{u_{n,r}\}_{n\geq 1}\subset U_\mu$ such that $\{u_{n,r}\}_{n\geq 1}$ is a (PS) sequence of $E_{r,\mu}(\cdot,\G)$ at level $c_r>0$.
Moreover, assume that $r \mapsto c_r$ is nondecreasing.
We first show that the (PS) sequences are bounded.
\begin{lemma}\label{lembounded}
    For $r>1$ sufficiently large, the (PS) sequence $\{u_{n,r}\}_{n\geq 1}$ at level $c_r$ is bounded in $H^1_A(\G,\C)$.
\end{lemma}
\begin{proof}
Since $\{u_{n,r}\}_{n\geq 1}$ is a (PS) sequence at level $c_r$, we have
$$
\begin{aligned}
 pc_r + o_n(1) +o_{n}(1)\norm{u_{n,r}} &= p E_{r, \mu}(u_{n,r},\G)- \langle  E_{r, \mu}^{\prime}(u_{n,r},\G), u_{n,r}\rangle\\
& =\frac{p-2}{2}\norm{u_{n,r}}^2-pf_r\left(\frac{\norm{u_{n,r}}_2^2}{\mu}\right)+2f'_r\left(\frac{\norm{u_{n,r}}_2^2}{\mu}\right)\frac{\norm{u_{n,r}}^2_{2}}{\mu}.
\end{aligned}
$$
Since, by \eqref{eqfr's}, $f_r^{\prime}(s)s > rf_r(s)$ and $f_r(s)>0$ for all $0 < s < 1$, it follows that, for $r \geq p/2$,
\begin{equation*}%
pc_r + o_n(1) + o_{n}(1)\norm{u_{n,r}} \geq \frac{p-2}{2}\norm{u_{n,r}}^2,
\end{equation*}
from which we conclude.
\end{proof}
Now, let us define
$$
\lambda_{n,r} := \frac{2}{\mu}f'_r\left(\frac{\norm{u_{n,r}}_2^2}{\mu}\right),
\qquad
\lambda_{\infty,r}:= \limsup_{n \rightarrow+ \infty} \lambda_{n,r},
$$
and, since $r \mapsto c_r$ is nondecreasing and  $c_{r} >0$ for $r>1$,
\[
c_\infty:=\lim\limits_{r\to+\infty} c_r>0.
\]
We have
\begin{lemma}\label{lemlambdanr}
If $c_\infty<+\infty$, then, for $r>1$ sufficiently large, $\lambda_{\infty,r} < +\infty$, and
$$
\limsup_{r \rightarrow+ \infty}\lambda_{\infty,r} \leq \frac{2c_\infty}{\mu}.
$$
\end{lemma}
\begin{proof}
For every $r>1$, since $f_{r}^{\prime}(0)=f_{r}(0)=0$ and $f_{r}^{\prime}(s) s-f_{r}(s) \to +\infty$ as $s \to 1^{-}$, by continuity, there exists $\xi_{r} \in (0,1)$ such that
$$
c_{\infty}=f_{r}^{\prime}(\xi_{r})\xi_{r}-f_{r}(\xi_{r}).
$$
We now claim that $\xi_{r}\to 1^{-}$ as $r \to +\infty$. Argue by contradiction and suppose that there exists a subsequence $\{r_{n}\}$ with $r_{n} \nearrow +\infty$ (monotone sequence) such that $\xi_{r_n}\to \xi$ as $n \to +\infty$ for some $\xi \in [0,1)$. Then, 
$$c_{\infty}=f_{r_n}^{\prime}(\xi_{r_n})\xi_{r_n}-f_{r_n}(\xi_{r_n})=\frac{(r_n-1) \xi_{r_n}^{r_n}}{1-\xi_{r_n}}+\frac{\xi_{r_n}^{r_n+1}}{(1-\xi_{r_n})^{2}} \to 0\quad \text{as } n\to +\infty,$$
which contradicts $c_{\infty}>0$. %
\\
Then, by \eqref{eqfr's}, we have 
$$
\frac{c_{\infty}}{r-1}=\frac{f_{r}^{\prime}(\xi_{r})\xi_{r}-f_{r}(\xi_{r})}{r-1}>f_{r}(\xi_{r})>0,
$$
so that
$f_{r}(\xi_{r}) \to 0$ as $r \to +\infty$, and we have also that
\begin{equation}\label{eqsupfr}
\lim_{r \to +\infty}f_{r}^{\prime}(\xi_{r})
= \lim_{r \to +\infty} \frac{f_{r}^{\prime}(\xi_{r})\xi_{r}-f_{r}(\xi_{r})+f_{r}(\xi_{r})}{\xi_{r}}
= \lim_{r \to +\infty} \frac{c_{\infty}+f_{r}(\xi_{r})}{\xi_{r}}
=c_{\infty}.
\end{equation}
Next, since $\{u_{n,r}\}_{n\geq 1}$ is a (PS) sequence at level $c_r$ and $p>2$, by Lemma \ref{lembounded}, one has
$$
\begin{aligned}
 o_{n}(1) &= \langle E_{r, \mu}^{\prime}(u_{n,r},\G), u_{n,r}\rangle\\
& =2 E_{r, \mu}(u_{n,r},\G)-(p-2) \Psi(u_{n,r},\G)+2f_r\left(\frac{\norm{u_{n,r}}_2^2}{\mu}\right)-2f'_r\left(\frac{\norm{u_{n,r}}_2^2}{\mu}\right)\frac{\norm{u_{n,r}}^2_{2}}{\mu}\\
&\leq 2 \left[ E_{r, \mu}(u_{n,r},\G)
+ f_r\left(\frac{\norm{u_{n,r}}_2^2}{\mu}\right)
-f'_r\left(\frac{\norm{u_{n,r}}_2^2}{\mu}\right)\frac{\norm{u_{n,r}}^2_{2}}{\mu}\right].
\end{aligned}
$$
Then,
$$
\limsup_{n \to +\infty}\left[f'_r\left(\frac{\norm{u_{n,r}}_2^2}{\mu}\right)\frac{\norm{u_{n,r}}^2_{2}}{\mu}-f_r\left(\frac{\norm{u_{n,r}}_2^2}{\mu}\right)\right] 
\leq \lim_{n \to +\infty} E_{r, \mu}(u_{n,r},\G)
=c_{r}\leq c_{\infty}
=f_{r}^{\prime}(\xi_{r}) \xi_{r}-f_{r}(\xi_{r})
$$
and so, since, by \eqref{eqhrs}, $s \mapsto f_{r}^{\prime}(s) s-f_{r}(s)$ is strictly increasing on $[0,1)$, we have $$\limsup _{n \to +\infty}\frac{\norm{u_{n,r}}_2^2}{\mu} \leq \xi_{r}.$$
Thus, since by \eqref{eqfr''s}, $s \mapsto f'_{r}(s)$ is strictly increasing on $[0,1)$, using \eqref{eqsupfr}, we get
$$
\lambda_{\infty,r}=\frac{2}{\mu} \limsup _{n \to +\infty} f'_r\left(\frac{\norm{u_{n,r}}_2^2}{\mu}\right) \leq \frac{2}{\mu} f_{r}^{\prime}\left(\xi_{r}\right)<+\infty
$$
for $r>1$ large enough and
$$
\limsup _{r \to +\infty} \lambda_{\infty,r} \leq \lim_{r \to +\infty}\frac{2}{\mu} f_{r}^{\prime}\left(\xi_{r}\right) = \frac{2c_{\infty}}{\mu}.
$$
\end{proof}
Now we want to prove the compactness of $\{u_{n,r}\}_{n\geq 1}$.
First observe that, using for instance Theorem \ref{thminmiax}, since $V\geq 1$, then $\inf\sigma_{\rm ess}(D^2_A + V(x))\geq 1$ (if $\sigma_{\rm ess}(D^2_A + V(x))=\emptyset$ we set $\inf\sigma_{\rm ess}(D^2_A + V(x))=+\infty$).

Assume
\begin{equation}\label{cmu2se}
c_\infty < \frac{\mu}{2} \inf\sigma_{\rm ess}(D^2_A + V(x))
\end{equation}
and define
$$
Y:=\operatorname{span}\left\{u \in H^1_A(\G,\C):(D^2_A + V(x)) u=\lambda u \text{ for some }\lambda \leq \frac{2c_\infty}{\mu}\right\}
$$
whose dimension, by \eqref{cmu2se}, is finite. Consequently,
$H^1_A(\G,\C)$ possesses the orthogonal decomposition
\begin{equation*}%
H^1_A(\G,\C) = Y \oplus Y^\perp.
\end{equation*}
Moreover, for any $u\in Y$ and $v \in Y^\perp$, $(u,v)_2=0$.

Let $P$ denote the orthogonal projector from $H^1_A(\G,\C)$ onto $Y$. Then, every $u \in H^1_A(\G,\C)$ can be written as
\begin{equation*}
    u = Pu + (I-P)u,
\end{equation*}
where $Pu \in Y$ and $(I-P)u \in Y^\perp$. Moreover, let $\delta >0$ be small enough such that $2{c_\infty}/\mu+\delta < \inf\sigma_{\rm ess}(D^2_A + V(x))$. Applying the same orthogonal decomposition as before for
$$
Y_\delta:=\operatorname{span}\left\{u \in H^1_A(\G,\C):(D^2_A + V(x)) u=\lambda u \text{ for some }\lambda \leq \frac{2{c_\infty}}{\mu}+\delta\right\}
$$
we have that the dimension of  $Y_\delta$ is finite. Then, if $\delta$ is sufficiently small that there are no $u\in Y_\delta\setminus\{0\}$ such that $(D^2_A + V(x)) u=\lambda u$
for $\lambda\in(2{c_\infty}/\mu, 2{c_\infty}/\mu + \delta]$, we have $Y_{\delta}=Y$.
Moreover, for any $v\in Y^\perp=Y^\perp_{\delta}$, we have
\begin{equation}\label{eqyperp}
\norm{v}^2 \geq \left(\frac{2{c_\infty}}{\mu}+\delta\right)\norm{v}^2_2.
\end{equation}
Indeed, if $Y=\{0\}$, then $\inf\sigma(D_A^2+V(x))>2c_\infty/\mu+\delta$, and hence \eqref{eqyperp} holds.
If $Y\ne\{0\}$, assume by contradiction that there exists $v \in Y^\perp\backslash\{0\}$ and 
$$
\norm{v}^2 < \left(\frac{2{c_\infty}}{\mu}+\delta\right)\norm{v}^2_2.
$$
Then, for any $u \in Y_\delta\backslash\{0\}$, we have 
\begin{equation}\label{equ+v}
\frac{\|u+v\|^2}{\|u+v\|_2^2}
= \frac{\|u\|^2+\|v\|^2}{\|u\|_2^2+\|v\|_2^2}
\leq \max\left\{\frac{\|u\|^2}{\|u\|_2^2},\frac{\|v\|^2}{\|v\|_2^2}\right\}
\leq \frac{2{c_\infty}}{\mu} + \delta.
\end{equation}
Analogously we can see that, for any $u,w\in  Y_\delta\backslash\{0\}$ with $u,w$ eigenfunctions corresponding to different eigenvalues,
we have  
$$
\frac{\|u+w\|^2}{\|u+w\|_2^2}
\leq \max\left\{\frac{\|u\|^2}{\|u\|_2^2},\frac{\|w\|^2}{\|w\|_2^2}\right\}
\leq \frac{2{c_\infty}}{\mu} + \delta.
$$
Then, since $\operatorname{dim}Y_\delta<+\infty$,
$$
 \sup_{\substack{u \in Y_\delta\setminus \{0\}}}\frac{\|u\|^2}{\|u\|_2^2}
 \leq \frac{2{c_\infty}}{\mu} + \delta,
$$
and, by \eqref{equ+v},
$$
 \sup_{\substack{u \in (Y_\delta \oplus \{tv : t \in \C\})\setminus \{0\}}}\frac{\|u\|^2}{\|u\|_2^2}
 \leq \frac{2{c_\infty}}{\mu} + \delta.
$$

Let us consider $Y_\delta$ as a complex space and set $j:=\operatorname{dim}Y_\delta$. Since $v\in Y_\delta^\perp\backslash\{0\}$, then $\operatorname{dim}(Y_\delta\oplus\{tv:t\in\C\})=j+1$.
Since $\operatorname{dom}(T)$ is dense in $H_A^1(\G,\C)$ with respect to the norm $\norm{\cdot}$, the min--max characterization in Theorem \ref{thminmiax} also holds over finite-dimensional complex subspaces of $H_A^1(\G,\C)$. Writing $\widehat{\mathcal M}_{j+1}$ for the family of such subspaces of dimension $j+1$, we obtain
$$
\widehat\lambda_{j+1}(T)
=\inf_{M\in\widehat{\mathcal M}_{j+1}}
\sup_{u\in M\setminus\{0\}}\frac{\|u\|^2}{\|u\|_2^2}
\leq\frac{2c_\infty}{\mu}+\delta
<\inf\sigma_{\rm ess}(T).
$$
Thus there are at least $j+1$ eigenvalues, counted with multiplicity, less than or equal to $2c_\infty/\mu+\delta$. This contradicts $\dim Y_\delta=j$.

In the next lemma, under the assumption \eqref{cmu2se}, we prove the compactness of (PS) sequences. We give the proof when $\G$ is a noncompact metric graph with a non-empty compact core $\K$. If $\G$ is compact the proof is the same taking $\K=\G$.
\begin{lemma}\label{lemunr}
If \eqref{cmu2se} holds, then, for $r>1$ sufficiently large, there exists $u_r \in U_\mu$ such that, up to a subsequence, $u_{n,r} \to u_r$ in $H^1_A(\G,\C)$ as $n \to +\infty$.  Moreover, $u_r$ satisfies $$
E_{r, \mu}(u_{r},\G)=c_{r} \text{ and } E_{r, \mu}^{\prime}(u_{r},\G)=0
$$
with 
$$
 \frac{2}{\mu}f'_r\left(\frac{\norm{u_r}_2^2}{\mu}\right) = \lambda_{\infty,r}  \text{ and } \limsup_{r \to +\infty}\lambda_{\infty,r} \leq \frac{2c_\infty}{\mu}.
$$
\end{lemma}
\begin{proof}
{Take $\delta>0$ as in \eqref{eqyperp}. By Lemma \ref{lemlambdanr}, for $r>1$ sufficiently large, up to a subsequence, we may assume that}
\begin{equation}\label{lnd}
0\leq\lambda_{n,r} \to \lambda_{\infty,r}
<\frac{2{c_\infty}}{\mu}+\delta 
\text{ as }n \to +\infty.
\end{equation}
By Lemma \ref{lembounded}, $\{u_{n,r}\}_{n\geq 1}$ is bounded in $H^1_A(\G,\C)$, and thus, up to a subsequence, $u_{n,r} \rightharpoonup u_{r}$ in $H^1_A(\G,\C)$.\\
Since $\{u_{n,r}\}_{n \geq 1} \subset U_\mu$, by Fatou's Lemma and using Remark \ref{recompact}, we get 
\begin{equation}\label{Fatmu}
\norm{u_{r}}_2^2
\leq \liminf_{n \to +\infty}\norm{u_{n,r}}_2^2
\leq \limsup_{n \to +\infty}\norm{u_{n,r}}_2^2
\leq\mu.
\end{equation}
We claim that $\norm{u_{r}}^2_2<\mu$. Arguing by contradiction, suppose that $\norm{u_{r}}_2^2= \mu$. Then, by \eqref{Fatmu},
\begin{equation}\label{unr2mu}
\lim_{n \to +\infty}\norm{u_{n,r}}_2^2= \mu.
\end{equation}
Since $\{u_{n,r}\}_{n \geq 1}$ is bounded  (and so, by Lemma \ref{lemgnsg}, $\{E(u_{n,r},\G)\}_{n \geq 1}$ is bounded) and, by \eqref{unr2mu},  $f_r(\norm{u_{n,r}}_2^2/{\mu})\to +\infty$, we obtain $E_{r,\mu}(u_{n,r},\G) \rightarrow -\infty$, which is absurd. 
    
Next, we are going to prove that $u_{n,r} \to u_{r}$ in $H_A^1(\G,\C)$.
To this end, using Remark \ref{recompact}, the compactness of the embedding $H^1_A(\G,\C)\hookrightarrow L^p(\K,\mathbb{C})$, and Lemma \ref{lemlambdanr}, it follows that
$$\langle E'(u_{r},\G), \cdot \rangle - \lambda_{\infty,r}(u_r,\cdot)_2=0$$
and 
\begin{equation*}
    \operatorname{Re}\int_{\K} (\abs{u_{n,r}}^{p-2} u_{n,r}-\abs{u_{r}}^{p-2} u_{r})\overline{(u_{n,r}-u_{r})} \, dx \to 0, \quad \text{ as }  n \to +\infty.
\end{equation*}
 Since $\langle E_{r,\mu}'(u_{n,r}),u_{n,r}-u_r\rangle=o_n(1)$, the above information gives
\begin{equation}\label{equnr}
\begin{aligned}
    \norm{u_{n,r}-u_{r}}^{2}
    &
    =\langle E'(u_{n,r},\G)-E'(u_r,\G), u_{n,r}-u_r \rangle \\
    &\qquad
    +\operatorname{Re}\int_{\K} (\abs{u_{n,r}}^{p-2} u_{n,r}-\abs{u_{r}}^{p-2} u_{r})\overline{(u_{n,r}-u_{r})} \, dx\\
    &=\langle E'_{r,\mu}(u_{n,r},\G), u_{n,r}-u_r \rangle -\langle E'(u_r,\G), u_{n,r}-u_r \rangle+\lambda_{n,r}(u_{n,r},u_{n,r}-u_r)_2
    \\
    &\qquad
     +\operatorname{Re}\int_{\K} (\abs{u_{n,r}}^{p-2} u_{n,r}-\abs{u_{r}}^{p-2} u_{r})\overline{(u_{n,r}-u_{r})} \, dx\\
    &=o_n(1)
    -\lambda_{\infty,r}(u_r,u_{n,r}-u_r)_2
    +\lambda_{n,r}(u_{n,r},u_{n,r}-u_r)_2\\
    &=
     \lambda_{\infty,r}\|u_{n,r}-u_r\|_2^2
    +(\lambda_{n,r}-\lambda_{\infty,r})(u_{n,r},u_{n,r}-u_r)_2
    +o_n(1)\\
    &=
    \lambda_{\infty,r}\|u_{n,r}-u_r\|_2^2
    +o_n(1).
\end{aligned}
\end{equation}
If $Y=\{0\}$, then $P=0$, and hence $P(u_{n,r}-u_r)=0$ in $H^1_A(\G,\C)$. In the nontrivial case, when the sequence $\{u_{n,r}\}$ is not constantly equal to $u_r$, $u_{n,r}-u_r\in Y^\perp$ and so, by \eqref{equnr} and \eqref{eqyperp}
$$
\lambda_{\infty, r}\norm{u_{n,r}-u_{r}}^{2}_2 + o_n(1) = \norm{u_{n,r}-u_{r}}^{2} \geq \left(\frac{2 c_{\infty}}{\mu}+\delta\right)\norm{u_{n,r}-u_{r}}^{2}_2.
$$
Hence, by \eqref{lnd}, we can conclude.
If, instead, $Y\neq\{0\}$, since the dimension of $Y$ is finite, we obtain $P(u_{n,r}-u_r) \to 0$ in $H^1_A(\G,\C)$. 
Then, arguing as before,
it follows that
$$
\lambda_{\infty, r}\norm{(I-P)(u_{n,r}-u_{r})}^{2}_2 + o_n(1) = \norm{(I-P)(u_{n,r}-u_{r})}^{2} \geq \left(\frac{2 c_{\infty}}{\mu}+\delta\right)\norm{(I-P)(u_{n,r}-u_{r})}^{2}_2.
$$
Hence, by \eqref{lnd}, $(I-P)(u_{n,r}-u_{r}) \to 0$ in $H^1_A(\G,\C)$. Therefore, $u_{n,r}\to u_{r}$ in $H^1_A(\G,\C)$ and so we can conclude.
\end{proof}
Let $\{r_{n}\}$ be a monotone sequence with $r_{n}>1$ and $r_{n} \nearrow +\infty$. Now, we show that, under the assumption \eqref{cmu2se}, the critical points $u_{r_n}$ of $E_{r_n,\mu}(\cdot,\G)$, obtained by Lemma \ref{lemunr}, converge, up to a subsequence, in $H^1_A(\G,\C)$ as $n\to +\infty$.
\begin{lemma}\label{lemur}
If \eqref{cmu2se} holds, then there exists $u \in H^1_A(\G,\C)$ such that, up to a subsequence, $u_{r_n} \to u$ in $H^1_A(\G,\C)$, as $n \to +\infty$. Moreover, $u$ satisfies $E(u,\G)=c_\infty$, and either $u$ is a critical point of $E(\cdot, \G)$ constrained on $H_\mu(\G)$ with Lagrange multiplier $\lambda \in [0,2c_\infty/\mu]$, or $u$ is a critical point of $E(\cdot, \G)$ constrained on $H_\nu(\G)$ for some $0<\nu < \mu$ with Lagrange multiplier $\lambda =0$.
\end{lemma}
\begin{proof}

By repeating the arguments of Lemmas \ref{lembounded} and \ref{lemunr}, we may assume that, up to a subsequence, $u_{r_n} \to u$ in $H^1_A(\G,\C)$ and there exists a $\lambda\in[0,2c_\infty/\mu]$ such that $\lambda_{\infty,r_n} \rightarrow \lambda$ as $n \to +\infty$.

We claim that
\begin{equation}\label{frn}
f_{r_n}\left(\frac{\norm{u_{r_n}}_2^2}{\mu}\right) \to 0
\text{ as }
n \to +\infty.
\end{equation}
Indeed, if $\norm{u}_2 = 0$, then \eqref{frn} follows by definition; if $0<\norm{u}_2^2\leq\mu$, it is an immediate consequence of \eqref{eqfr's}, since
$$
 \frac{2}{\mu}f'_{r_n}\left(\frac{\norm{u_{r_n}}_2^2}{\mu}\right) =\lambda_{\infty,r_n}  \text{ and } \limsup_{n \to +\infty}\lambda_{\infty,r_n} \leq \frac{2c_\infty}{\mu}.
$$
Hence, since $E_{r_{n}, \mu}(u_{r_{n}},\G)=c_{r_n}$ and $ E_{r_{n}, \mu}^{\prime}(u_{r_{n}},\G)=0$,
$$
E(u,\G) = c_\infty,\quad\langle  E^{\prime}(u,\G), \cdot\rangle-\lambda(u,\cdot)_2 = 0,\quad\norm{u}^2_{2} \leq \mu,\quad0 \leq \lambda \leq \frac{2c_\infty}{\mu}.$$
Since $c_\infty>0$, then $u$ is nontrivial, so that either $\norm{u}^2_{2}=\mu$ or $\norm{u}_{2}^2<\mu$. In  the latter case,
$$
0 
\leq \lambda
=\lim_{n\to +\infty}\lambda_{\infty,r_n}
=\lim _{n \to +\infty}\frac{2}{\mu}f'_{r_n}\left(\frac{\norm{u_{r_n}}_2^2}{\mu}\right) \leq \limsup _{n \to +\infty} \frac{2}{\mu}f_{r_{n}}^{\prime}\left(\frac{\mu+\norm{u}_{2}^{2}}{2\mu}\right)=0,
$$
which implies $\lambda=0$.
\end{proof}
\begin{remark}\label{remarkD-s}
Since, for $s \in [0,\lambda_1)$,
$$
\norm{u}^2-s \norm{u}_2^2
=
\left(1-\frac{s}{\lambda_1}\right)\norm{u}^2 + \frac{s}{\lambda_1}\norm{u}^2-s \norm{u}_2^2
\geq \left(1-\frac{s}{\lambda_1}\right)\norm{u}^2,
$$
for all $s \in (-\infty,\lambda_1)$, on $H^1_A(\G,\C)$ we can define the equivalent norm 
$$
\left(\int_\G \abs{D_Au}^2+  \left(V(x) -s\right)\abs{u}^2\,dx\right)^\frac{1}{2}.
$$ 
Thus, Lemmas
\ref{lembounded}--\ref{lemur} apply to the operator $D_A^2+\left(V(x)-s\right)$ for all $s \in (-\infty,\lambda_1)$, and this fact plays an important role in the proof of Theorem \ref{th2}.
\end{remark}

We conclude this section by defining the functional $J_{r,\mu}(\cdot,\G): H^1_A(\G,\C) \to \mathbb{R}$ by
$$J_{r,\mu}(u,\G) :=
\begin{cases}
    \beta(E_{r,\mu}(u,\G)) &\text{if } u\in U_\mu,\\
    -1  & \text{otherwise,}
\end{cases}
$$
where $\beta \in C^\infty(\mathbb{R},\mathbb{R})$ satisfies $\beta = -1$ on $(-\infty,-1]$, $\beta(t)=t$ for $t\in[0,+\infty)$, $\beta(t)<0$, for $t\in(-1,0)$.

Observe that, by the definition of $f_r$ we have that for every $u\in \partial U_\mu$ there exists $\varepsilon = \varepsilon(u) > 0$ such that $E_{r,\mu}(v,\G)\leq -1$ for every $v \in U_\mu \cap B_\varepsilon (u)$.

Then, arguing as in \cite[Lemma A.1]{Alves}, 
$J_{r,\mu}( \cdot,\G)\in C^1(H^1_A(\G,\C),\R)$ and, if $u$ is a critical point of $J_{r,\mu}( \cdot,\G)$ with $J_{r,\mu}(u,\G) > 0$, then $u\in U_\mu$ is also a critical point of $E_{r,\mu}(\cdot,\G)$ at the same energy level. The same conclusion holds for (PS) sequences. Therefore, instead of $E_{r,\mu}(\cdot, \G)$, we can look for positive min-max levels of $J_{r,\mu}(\cdot, \G)$, when it is more convenient.
\section{Proofs of Theorems \ref{th2} and \ref{th3}}\label{secp23}

In this section, we prove  Theorems \ref{th2} and \ref{th3}. Let $p>2$, $r>1$, and $\mu>0$.
The following lemma shows that the functional $E_{r, \mu}(\cdot,\G)$ possesses the mountain pass geometry in $U_\mu$.

\begin{lemma}\label{lemmpg}
The functional $E_{r, \mu}(\cdot,\G)$ possesses the following properties:
\begin{enumerate}[label=\rm(\roman*),ref=\roman*]
\item There exist $\alpha, \rho>0$ such that $E_{r, \mu}(u,\G) \geq \alpha$ for $u \in  U_{\mu}$ with $\|u\|=\rho$;\label{lemmpg1}
\item There exists  $e \in U_{\mu}$ with $\|e\|>\rho$ such that $E_{r, \mu}(e,\G) <0$.\label{lemmpg2}
\end{enumerate}
\end{lemma}
\begin{proof}
(\ref{lemmpg1}) For any $u \in H^1_A(\G,\C)$ with $\|u\|=\rho<\sqrt{\lambda_1\mu}$,  we have
$$
\norm{u}_{2}^2 \leq \frac{1}{\lambda_1}\|u\|^{2}=\frac{\rho^{2}}{\lambda_1} <\mu.
$$
Thus, by the monotonicity of $f_{r}$ (see \eqref{eqfr's}) and Lemma \ref{lemgnsg}, we obtain
$$
E_{r, \mu}(u,\G) 
\geq \frac{1}{2}\|u\|^{2}-\Psi(u,\G)-f_{r}\left(\frac{\rho^{2}}{\mu\lambda_1}\right)
\geq
\rho^{2}
\left[\frac{1}{2}-C \rho^{p-2}-\frac{1}{\lambda_1\mu-\rho^2}\left(\frac{\rho^2}{\mu\lambda_1}\right)^{r-1}\right].
$$
Since $p>2$ and $r>1$, we can conclude.

(\ref{lemmpg2}) Choosing $u_{0} \in H^1_A(\G,\C)$ with $\norm{u_0}^2_2=\mu$, it is easy to verify that
$$
\lim _{t \rightarrow 1^{-}} E_{r, \mu}(tu_0,\G)=-\infty.
$$
Therefore, since $\rho^2<\lambda_1\mu$, it is enough to take $t_0\in(\rho/\sqrt{\lambda_1\mu},1)$ sufficiently close to $1$ to get $ E_{r, \mu}(t_0u_0,\G)<0$.
\end{proof}
Thus we can define the minimax value
$$
c_{r}:=\inf _{\gamma \in \Gamma_{r, \mu}} \max _{t \in[0,1]} E_{r, \mu}(\gamma(t),\G)>0,
$$
where
$$
\Gamma_{r, \mu} := \{\gamma \in C([0,1], U_{\mu}): \gamma(0)=0,\,  E_{r, \mu}(\gamma(1),\G)<0\}.
$$
For any $r>1$, the mountain pass geometry allows us to find a (PS) sequence $\{u_{n,r}\}_{n\geq1}$ at the level $c_r$.

Moreover, if $r_{1} \leq r_{2}$, we have $c_{r_{1}} \leq c_{r_{2}}$ because $E_{r_1, \mu}(\cdot,\G) \leq E_{r_2, \mu}(\cdot,\G)$ on $U_{\mu}$. Then we can define $c_{\infty}:=\sup\limits_{r>1} c_{r}=\lim\limits_{r\to + \infty} c_{r}$. 

Note that, for any $r >1$ and any $u \in H^1_A(\G,\C) \backslash\{0\}$ with $\norm{u}_2^2 = \mu$, we have 
$$
c_r
\leq \sup _{t \in[0,1)} E_{r, \mu}(tu,\G)
\leq \sup_{0\leq t\leq 1} E(tu,\G).
$$
Then, for all $u \in H^1_A(\G,\C) \backslash\{0\}$ with $\norm{u}_2^2 = \mu$,  we have
\begin{equation*}%
    c_{\infty} \leq \sup _{0 \leq t \leq 1} E(t u,\G).
\end{equation*}
Now, we provide an important upper bound estimate for $c_\infty$.
\begin{lemma}\label{lemcin}
If the condition \eqref{AssG} holds, then
$$c_{\infty} < \frac{\mu\lambda_1}{2} \leq \frac{\mu}{2} \inf\sigma_{\rm ess}(D^2_A + V(x)).$$
    Instead,  if the condition \eqref{AssGm} holds, then $$c_{\infty} \leq \frac{\mu\lambda_1}{2} < \frac{\mu}{2} \inf\sigma_{\rm ess}(D^2_A + V(x)).$$
\end{lemma}
\begin{proof}
Let $\varphi_1$ be the eigenvector corresponding to the first eigenvalue $\lambda_1$ and satisfy $\norm{\varphi_1}^2_2=\mu$. Then, for $t \geq  0$,
$$
g(t):=E(t\varphi_1,\G) = \frac{t^2\mu\lambda_1}{2} - t^p\Psi(\varphi_1,\G).
$$
If the condition \eqref{AssG} holds, then $\Psi(\varphi_1,\G) > 0$. We have $g'(t) = t\mu \lambda_1 - pt^{p-1}\Psi(\varphi_1,\G)$.\\
Set
$$
t_0 := \left(\frac{\mu \lambda_1}{p\Psi(\varphi_1,\G)}\right)^\frac{1}{p-2}.
$$
If $t_0 \geq 1$, then $g(t)$ is increasing for $0 \leq t \leq 1$, and thus
$$
c_\infty \leq  \sup _{0 \leq t\leq 1} E(t \varphi_1,\G) \leq E(\varphi_1,\G) = \frac{\mu\lambda_1}{2} - \Psi(\varphi_1,\G) < \frac{\mu\lambda_1}{2}.
$$
If $0<t_0<1$, then $t_0$ is the unique maximum point of  $g(t)$ in $0 \leq t \leq 1$, and thus
$$
c_\infty
\leq  \sup _{0 \leq t \leq 1} E(t \varphi_1,\G) 
= E(t_0 \varphi_1,\G) = \frac{t_0^2\mu\lambda_1}{2} - t_0^p\Psi(\varphi_1,\G) < \frac{\mu\lambda_1}{2}.
$$
Instead, if the condition \eqref{AssGm} holds, then,
$$
c_\infty \leq  \sup _{0 \leq t\leq 1} E(t \varphi_1,\G)= \sup _{0 \leq t<1}\left(\frac{t^2\mu\lambda_1}{2} - t^p\Psi(\varphi_1,\G)\right) \leq \frac{\mu\lambda_1}{2}.
$$
\end{proof}

\begin{remark}\label{liststh2}
By repeating the arguments of Lemma \ref{lemmpg} for $D^2_A + (V(x)-s)$ with $s \in (-\infty, \lambda_1)$, for every $r>1$ we can define the minimax value
$$
c_{r,s}:=\inf _{\gamma \in \Gamma_{r, \mu, s}} \max _{t \in[0,1]} E_{r, \mu,s}(\gamma(t),\G)>0,
$$
where, for every $u \in U_\mu$
$$
E_{r, \mu,s}(u,\G):=E_{r,\mu}(u,\G)-\frac{s}{2}\norm{u}_2^2
$$
and
$$
\Gamma_{r, \mu,s} := \{\gamma \in C([0,1], U_{\mu}): \gamma(0)=0,\,  E_{r, \mu,s}(\gamma(1),\G)<0\}.
$$
For every $r>1$, the mountain pass geometry allows us to find a (PS) sequence $\{u_{n,r,s}\}_{n\geq1}$ of $E_{r, \mu,s}$ at the level $c_{r,s}>0$.

Then, for all $u \in H^1_A(\G,\C) \backslash\{0\}$ with $\norm{u}_2^2 = \mu$,  we have
that $r>1\mapsto c_{r,s}$ is nondecreasing and \begin{equation*}%
    c_{\infty,s}:=\sup_{r>1}c_{r,s} \leq \sup _{0 \leq t \leq 1} \left(E(t u,\G)-\frac{s}{2}\norm{tu}_2^2\right).
\end{equation*}
Then, by repeating the arguments of Lemma \ref{lemcin}, we obtain that, 
if the condition \eqref{AssG} holds, then
$$c_{\infty,s} < \frac{\mu\lambda_1(D^2_A + V(x)-s)}{2}=\frac{\mu(\lambda_1-s)}{2} \leq 
\frac{\mu}{2}\left(\inf\sigma_{\rm ess}(D_A^2+V(x))-s\right)
=\frac{\mu}{2} \inf\sigma_{\rm ess}(D^2_A + V(x)-s),$$
and,  if the condition \eqref{AssGm} holds, then $$c_{\infty,s} \leq \frac{\mu\lambda_1(D^2_A + V(x)-s)}{2}=\frac{\mu(\lambda_1-s)}{2} 
<\frac{\mu}{2}\left(\inf\sigma_{\rm ess}(D_A^2+V(x))-s\right)
= \frac{\mu}{2} \inf\sigma_{\rm ess}(D^2_A + V(x)-s),$$
where
$$\lambda_1(D^2_A + V(x)-s):=\inf_{u \in H^1_A(\G,\C)\backslash\{0\}} \frac{\norm{u}^2-s\norm{u}_2^2}{\norm{u}_2^2}.
$$
\end{remark}

Thus, we are ready to prove Theorem \ref{th2}. 
\begin{proof}[Proof of Theorem \ref{th2}]
Assume that either condition \eqref{AssG} or \eqref{AssGm} holds.  Then,   by Remarks \ref{liststh2} and \ref{remarkD-s} one of the following two alternatives occurs:
\begin{enumerate}[label=\rm(\roman*),ref=\roman*]
\item \label{pfrh21} either there exists  $u \in H^1_A(\G,\C)$ such that $u$ is a critical point of $E(\cdot, \G)-s\norm{\cdot}^2_2/2$ constrained on $H_\mu(\G)$ with a Lagrange multiplier $\lambda_s \in [0,\lambda_1-s]$ (or $\lambda_s\in[0,\lambda_1-s)$
if \eqref{AssG} holds),
\item  \label{pfrh22} or there exists a $u \in H^1_A(\G,\C)$ such that $u$ is a critical point of $E(\cdot, \G)-s\norm{\cdot}^2_2/2$ constrained on $H_\nu(\G)$ for some $0<\nu < \mu$ with Lagrange multiplier $\lambda_s =0$.
\end{enumerate}
Note that $s+\lambda_{s} \in[s,\lambda_1]$ (or $s+\lambda_{s} \in [s, \lambda_1)$ if \eqref{AssG} holds). \\
Thus, if there exists no critical point $u \in H_A^1(\G,\C)$ of $E(\cdot, \G)$ constrained on $H_\mu(\G)$   with Lagrange multiplier $\lambda \in (-\infty,\lambda_1]$ ($\lambda \in (-\infty, \lambda_1)$ if \eqref{AssG} holds),
then, for all $s \in (-\infty,\lambda_1)$, case (\ref{pfrh22}) occurs, i.e., for all $s \in (-\infty,\lambda_1)$, $E(\cdot, \G)$ has a critical point $u \in H^1_A(\G,\C)$ constrained on $H_\nu(\G)$ for some $0<\nu < \mu$ and the corresponding Lagrange multiplier is $s$.

On the other hand, if there exists $s_0\in (-\infty,\lambda_1)$ such that $E(\cdot, \G)$ has no critical point $u \in H^1_A(\G,\C)$ constrained on $H_\nu(\G)$ for some $0<\nu < \mu$ with Lagrange multiplier $s_0$, then, for $s=s_0$, case (\ref{pfrh21}) occurs, i.e., there exists a $u \in H^1_A(\G,\C)$ such that $u$ is a critical point of $E(\cdot, \G)$ constrained on $H_\mu(\G)$ with Lagrange multiplier $
\lambda_{s_0}+s_0\in[s_0,\lambda_1]$ ($\lambda_{s_0}+s_0\in[s_0,\lambda_1)$ if \eqref{AssG} holds).
    This completes the proof of Theorem \ref{th2}.
\end{proof}

Moreover, using Lemma \ref{lemnonex}, we can provide the proof of Theorem \ref{th3}.
\begin{proof}[Proof of Theorem \ref{th3}]
Assume that \eqref{AssG} or \eqref{AssGm} holds. By Lemma \ref{lemur} and Lemma \ref{lemcin}, we know that there exists $u \in H^1_A(\G,\C)$ satisfying $E(u,\G)=c_\infty\leq \mu\lambda_1/2$, and one of the following two alternatives occurs:
\begin{enumerate}[label=\rm(\roman*),ref=\roman*]
\item \label{pfth31} either $u$ is a critical point of $E(\cdot, \G)$ constrained on $H_\mu(\G)$ with Lagrange multiplier $\lambda \in [0,\lambda_1]$ ($\lambda \in [0, \lambda_1)$ if \eqref{AssG} holds) 
\item \label{pfth32} or $u$ is a critical point of $E(\cdot, \G)$ constrained on $H_\nu(\G)$ for some $0<\nu < \mu$ with Lagrange multiplier $\lambda =0$. 
\end{enumerate}
Applying  Lemma \ref{lemnonex} for $c=\lambda_1/2$, we conclude that, for any $0<\mu\leq\mu_0:=\mu_{\lambda_1/2,p}$, case (\ref{pfth32}) cannot occur.

Moreover, if $2<p<6$, by Lemma \ref{lemnonex}, since $\mu_{\lambda, p}^*\to +\infty$ as $\lambda \to -\infty$, we have that for any $\mu >0$ and for any $\lambda<0$ and $\nu>0$ such that $0<\nu<\mu\leq \mu_{\lambda, p}^*$, there exists no $u \in H_\nu(\G)$ satisfying $\langle E'(u,\G),\cdot \rangle-\lambda(\cdot,u)_2 = 0$. Thus, case (\ref{th2case2}) of Theorem \ref{th2} cannot occur and so we complete the proof.
\end{proof}

\section{Proof of Theorem \ref{th4}}\label{secp4}
In this section we assume that \eqref{AssGm} holds. Our goal is to prove Theorem \ref{th4}.

Assume that $m\geq 2$ (the case $m=1$ is included in Theorem \ref{th3}).

In the assumption \eqref{AssGm} we assume that for every $j=1,\ldots,m$, $\lambda_j<\inf\sigma_{\rm ess}(D^2_A + V(x))$. Here we consider $j=2,\ldots,m$  ($j\geq 2$ if $m=+\infty$)\footnote{For short, from now on we will write only $2 \leq j \leq m$.} since the case $j=1$ is included in Theorem \ref{th3}.

For any $2 \leq j \leq m$, let $\varphi_j$ be an eigenfunction corresponding to the eigenvalue $\lambda_j$ and, if $r>1$, define
$$
Y_j:=\operatorname{span}\{u \in H^1_A(\G,\C):(D^2_A + V(x)) u=\lambda u \text{ for some }\lambda \leq \lambda_j\},
\qquad
Z_j:=\overline{\bigoplus_{l=j}^{+\infty}X_{j,l}},
$$
$$
B_{r,j}:=\{u\in Y_j  : \norm{u}\leq \rho_{r,j}\}, \quad N_{r,j}:=\{u\in Z_j : \norm{u}= \xi_{r,j}\},
$$
where 
$X_{j,j}:=\operatorname{span}\{\varphi_j\}$, $X_{j,l}:=\operatorname{span}\{\varphi_{j,l}\}$ for $l\geq j+1$, $\{\varphi_{j,l}\}_{l\geq j+1}$ is an orthogonal
basis of $Y_j^\perp$, and $\rho_{r,j}>\xi_{r,j}>0$.

Note that, by \cite[Lemma A.1]{Alves}, we know that, if $\{u_{n,r,j}\}_{n \geq 1}\subset H^1_A(\G,\C)$ is a (PS) sequence of $J_{r,\mu}( \cdot,\G)$ (defined in Section \ref{seccompact}) at a positive level, then $\{u_{n,r,j}\}_{n \geq 1} \subset U_\mu$ and it is also a (PS) sequence of $E_{r,\mu}(\cdot,\G)$ at the same energy level. Thus, using \cite[Theorem 3.5]{W} for $J_{r,\mu}(\cdot,\G)$, we have the following result.

\begin{lemma}\label{lemth3.5}
    For $r>1$ and $2 \leq j  \leq m$, let
\[
c_{r,j} := \inf_{\gamma \in \Gamma_{r,j}} \max_{u \in B_{r,j}} J_{r,\mu}(\gamma(u),\G),
\]
where $\Gamma_{r,j} := \{\gamma \in C(B_{r,j}, H^1_A(\G,\C)) : -\gamma(u)=\gamma(-u) \text{ for all } u \in B_{r,j}\text{ and } \gamma\left.\right|_{\partial B_{r,j}} = \operatorname{id}\}$.\\
If
$$
b_{r,j} := \inf_{\begin{subarray}{c} u \in Z_j \\ \norm{u}= \xi_{r,j}\end{subarray}} J_{r,\mu}(u,\G) > 0 > a_{r,j} := \max_{\begin{subarray}{c} u \in Y_j \\ \norm{u} = \rho_{r,j} \end{subarray}} J_{r,\mu}(u,\G),
$$
 then \( c_{r,j} \geq b_{r,j} >0 \), and, for $E_{r, \mu}(\cdot,\G)$, there exists a (PS) sequence $\{u_{n,r,j}\}_{n\geq1}$ satisfying
\begin{equation*}
 E_{r, \mu}(u_{n,r,j},\G) \rightarrow c_{r,j}, \quad E_{r,\mu}^{\prime}(u_{n,r,j},\G) \rightarrow 0\quad \text{as }n \to +\infty.
\end{equation*}
\end{lemma}
To apply Lemma \ref{lemur}, we first establish the necessary estimates as in Section \ref{secp23}.

For $r>1$ and $2 \leq j  \leq m$, set $\rho_{r,j}:= \sqrt{\mu\lambda_j }$. We have
\begin{lemma}\label{lemfountain}
For $r>1$ and $2 \leq j  \leq m$, we have $c_{r,j} \leq \mu \lambda_j/2$ and $a_{r,j}=-1$.\\
Moreover, for any $2 \leq j  \leq m$ and $\varepsilon>0$, there exist $\xi_{\varepsilon,j}>0$ and $r_{\varepsilon,j}>1$, such that, for all $r\geq r_{\varepsilon,j}$,
\begin{equation}\label{brje}
\widetilde{b}_{r,j}: = \inf_{\begin{subarray}{c} u \in Z_j \\ \norm{u}= \xi_{\varepsilon,j}\end{subarray}} E_{r,\mu}(u,\G)  \geq \frac{\mu\lambda_j}{2}-\frac{C_{p,\G}\mu^\frac{p}{2}\lambda_j^\frac{p-2}{4}}{p}-\varepsilon.
\end{equation}
\end{lemma}
\begin{proof}
For all $u \in Y_j$, we have $\norm{u}^2\leq \lambda_j \norm{u}^2_2$ and so $Y_j \cap U_\mu \subset B_{r,j}$ and
$$
\sup_{u \in Y_j \cap U_\mu}E_{r,\mu}(u,\G) \leq \frac{\mu\lambda_j}{2}.
$$
Thus, taking $\gamma=\operatorname{id}$ in $\Gamma_{r,j}$,
$$
c_{r,j}
\leq \sup_{u \in B_{r,j}}J_{r,\mu}(u,\G) \leq \max\left\{0, \sup_{u \in Y_j \cap U_\mu}E_{r,\mu}(u,\G)\right\} \leq \frac{\mu \lambda_j}{2}.
$$
Moreover, for all $u \in Y_j$, using again that $\norm{u}^2\leq \lambda_j \norm{u}^2_2$ we have that, if $\|u\|=\rho_{r,j}$, then $\norm{u}^2_2 \geq \mu$ and so $a_{r,j}=-1$.
Now let $k>1$. For every $u \in Z_j$ with $\norm{u}^2=(k-1)\mu\lambda_j/k$, we have 
    $$
     \norm{u}^2_2 \leq \frac{\norm{u}^2}{\lambda_j} = \frac{\mu(k-1)}{k},
    $$
and so, by Lemma \ref{lemgnsg} and since $f_r$ is increasing,
$$
E_{r,\mu}(u,\G)
\geq 
\frac{k-1}{k}\frac{\mu\lambda_j}{2}-\left(\frac{k-1}{k}\right)^\frac{p}{2}\frac{C_{p,\G}\mu^\frac{p}{2}\lambda_j^\frac{p-2}{4}}{p}- k\left(\frac{k-1}{k}\right)^r.
$$
Thus, for any $\varepsilon>0$, taking $k$ large enough, we can find $\xi_{\varepsilon,j}>0$ and $r_{\varepsilon,j}>1$, such that, for all $r\geq r_{\varepsilon,j}$, 
$$ 
\widetilde{b}_{r,j}\geq \frac{\mu\lambda_j}{2}-\frac{C_{p,\G}\mu^\frac{p}{2}\lambda_j^\frac{p-2}{4}}{p}-\varepsilon.
$$
\end{proof}
Now we are ready to prove Theorem \ref{th4}.
\begin{proof}[Proof of Theorem \ref{th4}]
Let $2\leq k\leq m$.
Since $\lambda_{1} < \lambda_2 <\ldots< \lambda_{{k-1}} < \lambda_{k}$, there exists $\widetilde{\mu}_{k}$ such that, for all $0<\mu<\widetilde{\mu}_{k}$,
\begin{equation}\label{chain}
\frac{\mu \lambda_1}{2}< \frac{\mu\lambda_{2}}{2}-\frac{C_{p,\G}\mu^\frac{p}{2}\lambda_{2}^\frac{p-2}{4}}{p} < \frac{\mu\lambda_{2}}{2}<\ldots <\frac{\mu\lambda_{k}}{2}-\frac{C_{p,\G}\mu^\frac{p}{2}\lambda_{k}^\frac{p-2}{4}}{p} <\frac{\mu\lambda_k}{2}.
\end{equation}
Moreover, if $0<\mu<\widetilde{\mu}_k$, by Lemma \ref{lemfountain}, if $\varepsilon>0$ is sufficiently small, for $r$ large enough, by \eqref{chain}, for every $2\leq i\leq k$, $\widetilde{b}_{r,i}>0$. Then, by the definition of $J_{r,\mu}(\cdot,\G)$, 
\[
\inf_{\begin{subarray}{c} u \in Z_i \\ \norm{u}= \xi_{\varepsilon,i}\end{subarray}} J_{r,\mu}(u,\G)
= \widetilde{b}_{r,i}>0.
\]
Thus we can take $\xi_{r,i}=\xi_{\varepsilon,i}$ and so $b_{r,i}=\widetilde{b}_{r,i}>0$.\\
In addition, by Lemma \ref{lemfountain}, for any $2 \leq i \leq k$, 
$c_{r,i} \leq \mu \lambda_i/2$ and $a_{r,i}=-1$ for all $r>1$.\\
Hence, by Lemma \ref{lemth3.5}, if $0<\mu<\widetilde{\mu}_k$, $\varepsilon>0$ is sufficiently small, and $r$ is large enough, $c_{r,i}\geq b_{r,i}$ and so, using \eqref{chain}, 
$$
\frac{\mu\lambda_{i}}{2}
\geq c_{r,i}
\geq b_{r,i}
=\widetilde{b}_{r,i}
> \frac{\mu\lambda_{{i-1}}}{2}.
$$
Since $\rho_{r,i}, B_{r,i}$ and $\Gamma_{r,i}$ are independent of $r>1$, fix $1<r_1\leq r_2$ large enough (such that $c_{r_1,i}>0$).
For any $\gamma\in\Gamma_{r_1,i}=\Gamma_{r_2,i}$, choose $u_*\in B_{r_1,i}$ such that
$$
J_{r_1,\mu}(\gamma(u_*),\G)
=\max_{u\in B_{r_1,i}}J_{r_1,\mu}(\gamma(u),\G)
\geq c_{r_1,i}>0.
$$
Then $\gamma(u_*)\in U_\mu$ and $J_{r_1,\mu}(\gamma(u_*),\G)=E_{r_1,\mu}(\gamma(u_*),\G)>0$. Since $E_{r_1,\mu}(\cdot,\G)\leq E_{r_2,\mu}(\cdot,\G)$ on $U_\mu$, we have
\[
\max_{u\in B_{r_1,i}}J_{r_1,\mu}(\gamma(u),\G)
=E_{r_1,\mu}(\gamma(u_*),\G)
\leq E_{r_2,\mu}(\gamma(u_*),\G)
=J_{r_2,\mu}(\gamma(u_*),\G)
\leq\max_{u\in B_{r_2,i}}J_{r_2,\mu}(\gamma(u),\G).
\]
Taking the infimum over $\gamma\in\Gamma_{r_1,i}=\Gamma_{r_2,i}$, we obtain
\begin{equation}\label{moncr}
0<c_{r_1,i}\leq c_{r_2,i}
\end{equation}
and, due to the monotonicity in \eqref{moncr}, we can define 
$$c_{\infty,i}:=\lim_{r \to +\infty}c_{r,i}.$$
Thus, if we denote $c_{\infty,1}:=c_\infty$ which is defined in Section \ref{secp23},
using also Lemma \ref{lemcin}, we conclude
$$
    0 < c_{\infty,1} \leq \frac{\mu\lambda_{1}}{2} < b_{r,2} \leq c_{\infty,2} \leq \frac{\mu\lambda_{2}}{2} <\ldots <
    b_{r,k}\leq c_{\infty,k} \leq \frac{\mu\lambda_{k}}{2}.
$$

    Since $\lambda_{k} < \inf\sigma_{\rm ess}(D^2_A + V(x))$, using \eqref{moncr} and Lemma \ref{lemth3.5}, by Lemma \ref{lemur} we conclude that, for any $2 \leq i \leq k$, there exists $u_{i} \in H_A^1(\G,\C)$ satisfying $E(u_{i},\G)=c_{\infty,i}$ and one of the following two alternatives occurs:
     \begin{enumerate}[label=\rm(\roman*),ref=\roman*]
        \item \label{pfth41} either $u_i$ is a critical point of $E(\cdot, \G)$ constrained on $H_\mu(\G)$ with Lagrange multiplier $\omega_i \in [0,2c_{\infty,i}/\mu]\subset[0,\lambda_i]$
        \item \label{pfth42} or $u_i$ is a critical point of $E(\cdot, \G)$ constrained on $H_\nu(\G)$ for some $0<\nu < \mu$ with Lagrange multiplier $\omega_i =0$. 
    \end{enumerate}
    Applying Lemma \ref{lemnonex} for $c=\lambda_k/2$, we obtain that, if $\mu$ is small enough ($\mu\in(0,\min\{\mu_{\lambda_k/2,p},\widetilde{\mu}_k\})$), for any $1\leq i\leq k$, (\ref{pfth42}) cannot occur. Thus, $E( \cdot,\G)$ has at least $k$ critical points $u_1, u_2,\ldots,u_k$ constrained on $H_\mu(\G)$.
    
Now, let $2<p<6$ (or $p = 6$ with $\mu<\sqrt{3}C_{6,\G}^{-\frac{1}{2}}$)\footnote{Note that $\sqrt{3}C_{6,\G}^{-\frac{1}{2}}>C_{6,\G}^{-\frac{1}{2}}=\mu_{\lambda_{j_i}/{2},6}$.}, $m=+\infty$, and $\lim\limits_{j\to+\infty}\lambda_j=+\infty$. 
For $i \in \mathbb{N}^+$, we can choose  $j_i \geq 2$ such that $i \mapsto j_i$ is strictly increasing, $j_i\to +\infty$ and $\lambda_{j_i} \to +\infty$ as $i \to +\infty$. Then, by Lemma \ref{lemfountain}, there exist $\xi_{1,i}>0$ and $r_i>1$ such that 
$$
\tilde{b}_{r_i,j_i} \geq \frac{\mu\lambda_{j_i}}{2}-\frac{C_{p,\G}\mu^\frac{p}{2}\lambda_{j_i}^\frac{p-2}{4}}{p}-1.
$$
Thus, $\tilde{b}_{r_i,j_i} \to +\infty$ as $i \to +\infty$, and so, by the definition of $J_{r,\mu}(\cdot,\G)$, taking $\xi_{r_i,j_i}=\xi_{1,i}$, for $i$ large enough we have ${b}_{r_i,j_i}=\tilde{b}_{r_i,j_i}>0$ and, passing to a subsequence if necessary,  we can consider $i \mapsto b_{r_i,j_i}$ strictly increasing. Since $c_{\infty,j_i}\geq b_{r_i,j_i}\to+\infty$ as $i\to+\infty$, passing to a further subsequence, we may assume that $c_{\infty,j_{i-1}}<c_{\infty,j_i}$ for every $i\geq2$.

Hence, repeating the previous arguments, we conclude that, for any fixed $\mu>0$ (note that, since, by Lemma \ref{lemth3.5}, $c_{\infty,j_i}\geq b_{r_i,j_i} \to +\infty$ as $i \to +\infty$, the assumption that $\mu<\tilde{\mu}_k$ is not required) and for any $i \geq 2$, there exists $u_{i} \in H_A^1(\G,\C)$ satisfying $E(u_{i},\G)=c_{\infty,j_i}$ with $c_{\infty,j_{i-1}}<c_{\infty,j_i}\leq \mu\lambda_{j_i}/2$, and one of the following two alternatives occurs:
     \begin{enumerate}[label=\rm(\roman*),ref=\roman*]
        \item  either $u_i$ is a critical point of $E(\cdot, \G)$ constrained on $H_\mu(\G)$ with Lagrange multiplier $\omega_i \in [0,2c_{\infty,j_i}/\mu]\subset[0,\lambda_{j_i}]$
        \item  or $u_i$ is a critical point of $E(\cdot, \G)$ constrained on $H_\nu(\G)$ for some $0<\nu < \mu$ with Lagrange multiplier $\omega_i =0$. 
    \end{enumerate}

Then, by Lemma \ref{lemnonex}, we conclude that, for any
$0<\mu\leq\mu_{\lambda_1/2,p}=C_{p,\G}^{\frac{2}{2-p}}\lambda_1^{\frac{6-p}{2p-4}}$,
$E(\cdot,\G)$ has infinitely many critical points constrained on $H_\mu(\G)$ whose corresponding energy levels tend to $+\infty$.

Finally, let $2<p<6$, $m=+\infty$, and $\lim\limits_{j\to+\infty}\lambda_j=+\infty$, and let us extend the existence of infinitely many critical points as above to all $\mu>0$.
By Lemma \ref{lemnonex}, since $\mu_{\lambda, p}^*\to +\infty$ as $\lambda \to -\infty$, we have that for any $\mu,\nu >0$ and  $\lambda<0$ such that $0<\nu<\mu\leq \mu_{\lambda, p}^*$, there exists no $u \in H_\nu(\G)$ satisfying $\langle E'(u,\G),\cdot \rangle-\lambda(\cdot,u)_2 = 0$. Fix  $\lambda<0$ and $\mu\in(0, \mu_{\lambda, p}^*]$
and set
$$
E_\lambda(\cdot,\G):H_A^1(\G,\C)\to\mathbb{R},
\qquad
E_\lambda(u,\G):=E(u,\G)-\frac{\lambda}{2}\norm{u}_2^2.
$$
The operator $D_A^2+V(x)-\lambda$ satisfies \eqref{AssGm} with $m=+\infty$ and eigenvalues $\lambda_j-\lambda\to+\infty$ as $j \to +\infty$.
Repeating the previous construction for $E_\lambda$ we obtain results analogous to those in Lemmas \ref{lemth3.5} and \ref{lemfountain}. Then, using them (with $\varepsilon=1$) and applying Lemma \ref{lemur} as in Remark \ref{remarkD-s}, we conclude that, for every sufficiently large $i$, there exists a critical point $u_i\in H_A^1(\G,\C)$ at a suitable positive energy level satisfying
$$
\frac{\mu(\lambda_{j_i}-\lambda)}{2}
-\frac{C_{p,\G}\mu^{p/2}(\lambda_{j_i}-\lambda)^{(p-2)/4}}{p}
-1
\leq E_\lambda(u_i,\G)
\leq\frac{\mu(\lambda_{j_i}-\lambda)}{2},
$$
and so $E_\lambda(u_i,\G)\to+\infty$ as $i\to+\infty$. Moreover, one of the following two alternatives occurs:
\begin{enumerate}[label=\rm(\roman*),ref=\roman*]
\item either $u_i$ is a critical point of $E_\lambda(\cdot,\G)$ constrained on $H_\mu(\G)$ with Lagrange multiplier
$\omega_i\in\left[0,\frac{2E_\lambda(u_i,\G)}{\mu}\right]
\subset[0,\lambda_{j_i}-\lambda]$;
\item or $u_i$ is a critical point of $E_\lambda(\cdot,\G)$ constrained on $H_{\nu_i}(\G)$ for some $0<\nu_i<\mu$ with Lagrange multiplier $\omega_i=0$.
\end{enumerate}
The latter case would imply
$$
\langle E'(u_i,\G),\cdot\rangle-\lambda(u_i,\cdot)_2=0,
\qquad
0<\norm{u_i}_2^2<\mu\leq\mu_{\lambda,p}^*,
$$
which is excluded by Lemma \ref{lemnonex}.
Thus, $\norm{u_i}_2^2=\mu$ for every $i$. Therefore, each $u_i$ is also a critical point of $E(\cdot,\G)$ constrained on $H_\mu(\G)$, and
$
E(u_i,\G)=E_\lambda(u_i,\G)+\lambda\mu/{2}
\to+\infty$ as $i\to+\infty$.
 Hence $E(\cdot,\G)$ has infinitely many distinct critical points constrained on $H_\mu(\G)$. This completes the proof.

\end{proof}

\section{Proofs of Theorems \ref{thn3} and \ref{thn4}}
\label{secnoncompact}
In this section, under assumption (\ref{AssGm}), we consider
a connected noncompact metric graph $\G=(\V,\E)$ with a finite number of edges, whose compact core may be either empty or nonempty.

We shall study the existence and multiplicity of critical points of the functional $E(\cdot,\G)$ constrained on the $L^2$-sphere
$H_\mu(\G)$ with $\mu > 0$.
The proofs of Theorems \ref{thn3} and \ref{thn4} are almost the same as the proofs of Theorems \ref{th2} and \ref{th4}, respectively, except for Lemmas \ref{lemunr} and \ref{lemcin}. Therefore, we give the details for this preliminary part and we only present a sketch of the proofs of Theorems \ref{thn3} and \ref{thn4} here.

Using the same notation introduced before, for a fixed $\mu >0$, we assume that, for any $r >1$ sufficiently large, there exists a (PS) sequence $\{u_{n,r}\}_{n\geq 1}\subset H^1_A(\G,\C)$ for $E_{r,\mu}(\cdot,\G)$ at level $c_r>0$ and that $r \mapsto c_r$ is nondecreasing.

Since we are dealing with problem \eqref{eqmagnoncompact} on noncompact metric graphs, we need to overcome the lack of compactness. 

Recalling that, by \eqref{AssGm}, for $1\leq k\leq m$, $\lambda_k$ is an eigenvalue below $\inf\sigma_{\rm ess}(D^2_A + V(x))$, let 
 \begin{equation}\label{eqdeltak}
      \delta_k:=\begin{cases}
     \displaystyle\frac{\inf\sigma_{\rm ess}(D^2_A + V(x))-\lambda_k}{2}\quad &\text{if }\inf\sigma_{\rm ess}(D^2_A + V(x))<+\infty\\
     1 &\text{if }\inf\sigma_{\rm ess}(D^2_A + V(x))=+\infty,
 \end{cases}
 \end{equation}
and
$$
\widetilde{Y}_k:=\operatorname{span}\{u \in H^1_A(\G,\C):(D^2_A + V(x)) u=\lambda u \text{ for some }\lambda \leq \lambda_k+\delta_k\}.
$$ Then, since $\lambda_k + \delta_k<\inf\sigma_{\rm ess}(D^2_A + V(x))$, the dimension of $\widetilde{Y}_k$ is finite. Consequently, $H^1_A(\G,\C)$ possesses the orthogonal decomposition
\begin{equation*}%
     H^1_A(\G,\C) = \widetilde{Y}_k \oplus \widetilde{Y}_k^\perp,
\end{equation*}
so that for every $u\in \widetilde{Y}_k$ and $v \in \widetilde{Y}_k^\perp$, $(u,v)_2=0$.
Let $P_k$ denote the orthogonal projector from $H^1_A(\G,\C)$ onto $\widetilde{Y}_k$. Then, every $u \in H^1_A(\G,\C)$ can be written as
\begin{equation*}
    u = P_ku + (I-P_k)u,
\end{equation*}
where $P_ku \in \widetilde{Y}_k$ and $(I-P_k)u \in \widetilde{Y}_k^\perp$ and, arguing as in the proof of \eqref{eqyperp},
\begin{equation}\label{eqnyperp}
\norm{v}^2 \geq
(\lambda_k+\delta_k)\norm{v}_2^2 \quad  \text{ for any }v \in \widetilde{Y}_k^\perp.
\end{equation}
Set now
\begin{equation*}%
\mu^{**}_k:=C_{\infty,\G}^{-2}\left(\frac{p-2}{p\lambda_k}\right)^\frac{1}{2}\left[\frac{\delta_k}{3(p-1)}\right]^\frac{2}{p-2}
\end{equation*}
and
$$c_\infty:=\lim\limits_{r\to+\infty} c_r.
$$

\begin{lemma}\label{lemnur}
     If  $0<\mu\leq\mu^{**}_k$ and  $c_\infty \leq {\mu}\lambda_k/{2}$, then, for  $r>1$ sufficiently large, there exists $u_r \in U_\mu$ such that, up to a subsequence, $u_{n,r} \to u_r$ in $H^1_A(\G,\C)$, as $n \to +\infty$.  Moreover, $u_r$ satisfies $$
E_{r, \mu}(u_{r},\G)=c_{r} \text{ and } E_{r, \mu}^{\prime}(u_{r},\G)=0
$$
with 
$$
 \frac{2}{\mu}f'_r\left(\frac{\norm{u_r}_2^2}{\mu}\right) = \lambda_{\infty,r}  \text{ and } \limsup_{r \to +\infty}\lambda_{\infty,r} \leq \frac{2c_\infty}{\mu}.
$$
\end{lemma}
\begin{proof}
Take $\delta_k>0$ as in \eqref{eqdeltak}. Repeating the argument in Lemma \ref{lemlambdanr}, we obtain $\limsup_{r\to+\infty}\lambda_{\infty,r}\leq 2c_\infty/\mu$. Hence, for $r>1$ sufficiently large, we have
\begin{equation}\label{eqnlambdanr}
0\leq\lambda_{\infty,r}
<\frac{2c_\infty}{\mu}+\frac{\delta_k}{3}
\leq\lambda_k+\frac{\delta_k}{3}.
\end{equation}
Passing to a subsequence, we may assume that
$\lambda_{n,r}\to\lambda_{\infty,r}$ as $n\to+\infty$. Then, as in Lemma \ref{lembounded}, we can conclude that $\{u_{n,r}\}_{n\geq 1}$ is bounded  in $H^1_A(\G,\C)$, 
    \begin{equation}\label{eqnsupbound}
            \limsup_{n \to +\infty} \norm{u_{n,r}}^2\leq \frac{2pc_r}{p-2}\leq \frac{2pc_\infty}{p-2}\leq \frac{p\mu\lambda_k}{p-2},
    \end{equation}
    and thus, up to a subsequence, $u_{n,r} \rightharpoonup u_{r}$ in $H^1_A(\G,\C)$ and, arguing as in the proof of Lemma \ref{lemunr}, we get that $\norm{u_r}_2^2 <\mu$.

    Now we are going to prove that $u_{n,r} \to u_{r}$ in $H_A^1(\G,\C)$. To this end, we first verify that 
    \begin{equation}\label{eqneur}
        \langle E'(u_{r},\G), \cdot \rangle - \lambda_{\infty,r}(u_r,\cdot)_2=0.
    \end{equation} 
    Recall that, by Lemma \ref{lemdense}, $H^1_{A,c}(\G,\C)$ is a dense subset of $H^1_{A}(\G,\C)$. For any $\varphi \in H^1_{A,c}(\G,\C)$, since $\{u_{n,r}\}_{n\geq 1}$ is a (PS) sequence and $H^1_A(\G,\C)$ is compactly embedded in $L^2(B,\mathbb{C})$, where $B = \operatorname{supp}\varphi$ (see Remark \ref{recompact}),  we have
$$
\begin{aligned}
\langle E'(u_{r},\G), \varphi \rangle - \lambda_{\infty,r}(u_r,\varphi)_2=
     &\langle E'_{r,\mu}(u_{n,r},\G), \varphi \rangle+ (u_r-u_{n,r},\varphi) +\lambda_{n,r}(u_{n,r},\varphi)_2 - \lambda_{\infty,r}(u_r,\varphi)_2  \\&+\operatorname{Re}\int_{B} (\abs{u_{n,r}}^{p-2} u_{n,r}-\abs{u_{r}}^{p-2} u_{r})\overline{\varphi} \, dx\\
 =&(\lambda_{n,r}-\lambda_{\infty,r})(u_{n,r},\varphi)_2 +\lambda_{\infty,r}(u_{n,r}-u_r,\varphi)_2+o_n(1)\\
=&o_n(1).
\end{aligned}
$$
Thus, by density, we get \eqref{eqneur}.

Moreover
\begin{equation}\label{eqnunr}
\begin{split}
    \norm{u_{n,r}-u_{r}}^{2}
    &=
    \langle E'(u_{n,r},\G)-E'(u_r,\G), u_{n,r}-u_r \rangle \\
    &\quad
    +\operatorname{Re}\int_{\G} (\abs{u_{n,r}}^{p-2} u_{n,r}-\abs{u_{r}}^{p-2} u_{r})\overline{(u_{n,r}-u_{r})} \, dx\\
    &=
    \langle E'_{r,\mu}(u_{n,r},\G), u_{n,r}-u_r \rangle 
    -\langle E'(u_r,\G), u_{n,r}-u_r \rangle
    +\lambda_{n,r}(u_{n,r},u_{n,r}-u_r)_2 \\
    &\quad
    +\operatorname{Re}\int_{\G} (\abs{u_{n,r}}^{p-2} u_{n,r}-\abs{u_{r}}^{p-2} u_{r})\overline{(u_{n,r}-u_{r})} \, dx\\
    &=\lambda_{\infty, r}\norm{u_{n,r}-u_{r}}_{2}^{2} +\operatorname{Re}\int_{\G} (\abs{u_{n,r}}^{p-2} u_{n,r}-\abs{u_{r}}^{p-2} u_{r})\overline{(u_{n,r}-u_{r})} \, dx+o_{n}(1).\\
\end{split}
\end{equation}

Let $g:\R^2\to\R^2$ be defined by $g(z)=|z|^{p-2}z$.
Since $p>2$, we have $g\in C^1(\R^2,\R^2)$ and $Dg(0)=0$.
For $z\neq0$,
$$
Dg(z)=|z|^{p-2}I+(p-2)|z|^{p-4}z\otimes z,
$$
and hence
$$
|Dg(z)v|\leq(p-1)|z|^{p-2}|v|
\qquad\text{for all }z,v\in\R^2.
$$
Therefore, by the fundamental theorem of calculus, for all $z_1,z_2\in\R^2$,
$$
|g(z_2)-g(z_1)|\leq
(p-1)\max\{|z_1|^{p-2},|z_2|^{p-2}\}|z_2-z_1|.
$$
By the weak lower semicontinuity of the norm and \eqref{eqnsupbound}, we have
$\|u_r\|^2\leq p\mu\lambda_k/(p-2)$ and
$\|u_{n,r}\|^2\leq p\mu\lambda_k/(p-2)+o_n(1)$.
Since $\{u_{n,r}-u_r\}$ is bounded in $L^2(\G,\C)$, Lemma \ref{lemgnsg} gives
\begin{align*}
\left|\operatorname{Re}\int_{\G} (\abs{u_{n,r}}^{p-2} u_{n,r}-\abs{u_{r}}^{p-2} u_{r})\overline{(u_{n,r}-u_{r})} \, dx\right|
&\leq
(p-1)\max\{\norm{u_{n,r}}_\infty^{p-2},\norm{u_{r}}_\infty^{p-2}\} \|u_{n,r}-u_{r}\|_2^2 
\\
&\leq
(p-1)C_{\infty,\G}^{p-2}\mu^{\frac{p-2}{4}}\max\{\norm{u_{n,r}}^\frac{p-2}{2},\norm{u_{r}}^\frac{p-2}{2}\}\norm{u_{n,r}-u_{r}}_{2}^{2}\\
&\leq
(p-1)C_{\infty,\G}^{p-2}\mu^{\frac{p-2}{2}}\left(\frac{p\lambda_k}{p-2}\right)^\frac{p-2}{4}\norm{u_{n,r}-u_{r}}_{2}^{2}+o_n(1)
\end{align*}
and so, by \eqref{eqnunr}, we obtain
$$
\norm{u_{n,r}-u_{r}}^{2}
\leq
\left[\lambda_{\infty, r} + (p-1)C_{\infty,\G}^{p-2}\mu^{\frac{p-2}{2}}\left(\frac{p\lambda_k}{p-2}\right)^\frac{p-2}{4}\right]\norm{u_{n,r}-u_{r}}_{2}^{2}+o_{n}(1).
$$
Now, if $0<\mu \leq \mu^{**}_k$, we get
\begin{equation}\label{eqnunrmu}
    \norm{u_{n,r}-u_{r}}^{2} \leq\left(\lambda_{\infty, r}+ \frac{1}{3}\delta_k \right)\norm{u_{n,r}-u_{r}}_{2}^{2}+o_{n}(1).
\end{equation}
Since the dimension of $\widetilde{Y}_k$ is finite, we obtain $P_k(u_{n,r}-u_r) \to 0$ in $H^1_A(\G,\C)$. Moreover, by \eqref{eqnyperp} and  \eqref{eqnunrmu}  it follows that 
$$
\left(\lambda_{\infty,r}+\frac{\delta_k}{3}\right)
\norm{(I-P_k)(u_{n,r}-u_r)}_2^2+o_n(1)
\geq \norm{(I-P_k)(u_{n,r}-u_r)}^2
\geq
(\lambda_k+\delta_k)\norm{(I-P_k)(u_{n,r}-u_r)}_2^2.
$$
Hence, by \eqref{eqnlambdanr}, $(I-P_k)(u_{n,r}-u_{r}) \to 0$ in $H^1_A(\G,\C)$. Therefore, $u_{n,r}\to u_{r}$ in $H^1_A(\G,\C)$ and so we can conclude.
\end{proof}

Let $\{r_n\}$ be a monotone sequence with $r_n>1$ sufficiently large and $r_n\nearrow+\infty$, and for each $n$, let $u_{r_n}$ be a critical point of $E_{r_n,\mu}(\cdot,\G)$ obtained by Lemma \ref{lemnur}. As in Lemma \ref{lemur}, we now prove that $u_{r_n}$ converges, up to a subsequence, in $H_A^1(\G,\C)$ as $n\to+\infty$.

\begin{lemma}\label{lemnurinfty}
If $0<\mu\leq\mu_k^{**}$ and $c_\infty\leq\mu\lambda_k/2$, then there exists $u\in H^1_A(\G,\C)$ such that, up to a subsequence, $u_{r_n}\to u$ in $H^1_A(\G,\C)$, as $n\to+\infty$. Moreover, $u$ satisfies $E(u,\G)=c_\infty$, and, either $u$ is a critical point of $E(\cdot,\G)$ constrained on $H_\mu(\G)$ with Lagrange multiplier $\lambda\in[0,2c_\infty/\mu]$, or $u$ is a critical point of $E(\cdot,\G)$ constrained on $H_\nu(\G)$ for some $0<\nu<\mu$ with Lagrange multiplier $\lambda=0$.
\end{lemma}

\begin{proof}
By \eqref{eqnsupbound}, for $n$ sufficiently large,
$$
\norm{u_{r_n}}^2
\leq\frac{2pc_{r_n}}{p-2}
\leq\frac{p\mu\lambda_k}{p-2}.
$$
Thus, using also Lemma \ref{lemnur}, up to a subsequence, $u_{r_n}\rightharpoonup u$ in $H_A^1(\G,\C)$ and $\lambda_{\infty,r_n}\to\lambda\in[0,2c_\infty/\mu]$.
Arguing as in the proof of \eqref{eqneur}, we obtain $\langle E'(u,\G),\cdot\rangle-\lambda(u,\cdot)_2=0$.
Moreover, since $\mu\leq\mu_k^{**}$, the estimates leading to \eqref{eqnunrmu} give
\begin{equation}\label{eqnunrmug}
\norm{u_{r_n}-u}^2
\leq\left(\lambda+\frac{\delta_k}{3}\right)
\norm{u_{r_n}-u}_2^2+o_n(1).
\end{equation}
Since $\widetilde{Y}_k$ is finite-dimensional, $P_k(u_{r_n}-u)\to0$ in $H_A^1(\G,\C)$. Combining \eqref{eqnunrmug} with \eqref{eqnyperp} and using $\lambda+\delta_k/3<\lambda_k+\delta_k$, we conclude, as in the proof of Lemma \ref{lemnur}, that $u_{r_n}\to u$ in $H_A^1(\G,\C)$.

By \eqref{eqfr's} and Lemma \ref{lemnur},
$$
0\leq f_{r_n}\left(\frac{\norm{u_{r_n}}_2^2}{\mu}\right)
\leq \frac{1}{r_n}
f'_{r_n}\left(\frac{\norm{u_{r_n}}_2^2}{\mu}\right)
\frac{\norm{u_{r_n}}_2^2}{\mu}=\frac{\lambda_{\infty,r_n}}{2r_n}\norm{u_{r_n}}_2^2
\leq\frac{\mu\lambda_{\infty,r_n}}{2r_n}\to0.
$$
Arguing as in the proof of Lemma \ref{lemur}, we obtain
$$E(u,\G)=c_\infty>0,
\quad
\text{and}
\quad
\langle  E^{\prime}(u,\G), \cdot\rangle-\lambda(u,\cdot)_2 = 0,
$$
which imply that $u\neq0$, that $\lambda=0$ if $\norm{u}_2^2<\mu$, and that $u$ is a critical point of $E(\cdot,\G)$ constrained on $H_\mu(\G)$ with Lagrange multiplier $\lambda$ if $\norm{u}_2^2=\mu$.
\end{proof}

 Repeating the arguments used in the first part of Section \ref{secp23}, for $r > 1$, the mountain pass geometry of the functional $E_{r, \mu}(\cdot,\G)$ allows us to find a (PS) sequence $\{u_{n,r}\}_{n\geq1}$  satisfying 
\begin{equation*}
 E_{r, \mu}(u_{n,r},\G) \rightarrow c_{r}, \quad E_{r,\mu}^{\prime}(u_{n,r},\G) \rightarrow 0,
\end{equation*}
where $r\mapsto c_r$ is nondecreasing. Moreover, defining $$c_{\infty}:=\sup\limits_{r>1} c_{r}=\lim\limits_{r\to + \infty} c_{r},$$
we can show that, for all $u \in H^1_A(\G,\C) \backslash\{0\}$ with $\norm{u}_2^2 = \mu$, we have
\begin{equation*}%
    c_{\infty} \leq \sup _{0 \leq t<1} E(t u,\G).
\end{equation*}
In this case we have the following upper bound estimate for $c_\infty$, which is different from  Lemma \ref{lemcin} under condition \eqref{AssGm}, while its proof is similar to the argument of Lemma \ref{lemcin} under condition \eqref{AssG}.
\begin{lemma}\label{lemncin}
    If the condition \eqref{AssGm} holds, then
    $$c_{\infty} < \frac{\mu\lambda_1}{2} < \frac{\mu}{2} \inf\sigma_{\rm ess}(D^2_A + V(x)).$$
\end{lemma}
\begin{proof}
Let $\varphi_1$ be the eigenfunction corresponding to the first eigenvalue $\lambda_1$ and that satisfies $\norm{\varphi_1}^2_2=\mu$. Then,  for $t \geq  0$,
$$
g(t):=E(t\varphi_1,\G) = \frac{t^2\mu\lambda_1}{2} - t^p\Psi(\varphi_1,\G).
$$
The nontriviality of $\varphi_1$ implies that $\Psi(\varphi_1,\G)>0$ and so
$$t_0 := \left(\frac{\mu \lambda_1}{p\Psi(\varphi_1,\G)}\right)^\frac{1}{p-2}$$
is the unique critical point (maximizer) of $g$ in $(0,+\infty)$.
If $t_0 
\geq 1$, then $g$ is increasing for $0 \leq t \leq 1$, and thus
$$
c_\infty
\leq  \sup _{0 \leq t\leq 1} E(t \varphi_1,\G)
\leq E(\varphi_1,\G)
=\frac{\mu\lambda_1}{2} - \Psi(\varphi_1,\G) < \frac{\mu\lambda_1}{2}.
$$
If $0<t_0<1$, then $t_0$ is the unique maximum point of  $g(t)$ in $0 \leq t \leq 1$, and thus
$$
c_\infty 
\leq
\sup _{0 \leq t\leq 1} E(t \varphi_1,\G) = E(t_0 \varphi_1,\G) = \frac{t_0^2\mu\lambda_1}{2} - t_0^p\Psi(\varphi_1,\G) < \frac{\mu\lambda_1}{2}.
$$
\end{proof}

\begin{remark}\label{remckn}
Fix $2\leq k\leq m$ and $0<\mu<\widetilde{\mu}_k$.
For $r>1$ and $2\leq i\leq k$, define $c_{r,i}$ and $b_{r,i}$ as in Lemma \ref{lemth3.5}, with $Y_i$, $Z_i$, $\rho_{r,i}$, and $\xi_{r,i}$ chosen as in Section \ref{secp4}.
By Lemma \ref{lemgnsg}, the estimates in Lemma \ref{lemfountain} also hold for
$$\Psi(u,\G)=\frac{1}{p}\int_\G\abs{u}^p\,dx.$$
Using \eqref{chain}, take $\varepsilon>0$ sufficiently small in \eqref{brje} and set $\xi_{r,i}=\xi_{\varepsilon,i}$.
By Lemma \ref{lemfountain} and the argument in the proof of Theorem \ref{th4}, we have $b_{r,i}>0$ for $2\leq i\leq k$ and $r>1$ sufficiently large.
Lemma \ref{lemth3.5} therefore yields $c_{r,i}\geq b_{r,i}>0$.
By the argument proving \eqref{moncr} and Lemma \ref{lemfountain}, we obtain, for $1<r_1\leq r_2$ sufficiently large,
$$
0<c_{r_1,i}\leq c_{r_2,i}\leq\frac{\mu\lambda_i}{2}.
$$
Hence the limits
$$
c_{\infty,i}:=\lim_{r\to+\infty}c_{r,i},
\qquad 2\leq i\leq k,
$$
are well defined.

Taking $c_{\infty,1}:=c_\infty$ and using Lemma \ref{lemncin}, the argument in the proof of Theorem \ref{th4} gives, for $2\leq i\leq k$ and $r>1$ sufficiently large,
\begin{equation}\label{monchain}
c_{\infty,i-1}
\leq\frac{\mu\lambda_{i-1}}{2}
<b_{r,i}
\leq c_{r,i}
\leq c_{\infty,i}
\leq\frac{\mu\lambda_i}{2}.
\end{equation}
\end{remark}

Thus we are ready to prove Theorems \ref{thn3} and \ref{thn4}.

\begin{proof}[Proof of Theorems \ref{thn3} and \ref{thn4}]
As in the proof of Theorem \ref{th3}, using Lemma \ref{lemnurinfty} for $k=1$ and Lemma \ref{lemncin},  for $0<\mu \leq \mu^{**}_1$, there exists $u \in H^1_A(\G,\C)$ satisfying $E(u,\G)=c_\infty < \mu\lambda_1/2$, and one of the following two cases must hold:
     \begin{enumerate}[label=\rm(\roman*),ref=\roman*]
        \item \label{pfthn341} either $u$ is a critical point of $E(\cdot, \G)$ constrained on $H_\mu(\G)$ with Lagrange multiplier 
        $\lambda \in [0, \lambda_1)$
\item \label{pfthn342} or $u$ is a critical point of $E(\cdot, \G)$ constrained on $H_\nu(\G)$ for some $0<\nu < \mu$ with Lagrange multiplier $\lambda =0$.
    \end{enumerate}
Set $\mu_0^*:=\min\{\mu_{\lambda_1/2,p},\mu_1^{**}\}$ (see \eqref{mucp}). If $0<\mu\leq\mu_0^*$ and $0<\nu<\mu$, Lemma \ref{lemnonex}, applied with $c=\lambda_1/2$, excludes any $u\in H_\nu(\G)$ satisfying $E'(u,\G)=0$ and $E(u,\G)\leq\mu\lambda_1/2$. Hence case (\ref{pfthn342}) cannot occur. This completes the proof of Theorem \ref{thn3}.

Instead, to prove Theorem \ref{thn4}, we can argue as in Section \ref{secp4}. For brevity, let us consider $k=2,\ldots,m$. By Lemma \ref{lemnurinfty} and Remark \ref{remckn}, we get that, for all $0<\mu<\min\{\widetilde{\mu}_{k},\mu^{**}_k\}$ (as in Section \ref{secp4}, $0<\mu<\widetilde{\mu}_{k}$ is used to get the same estimate for $c_{\infty,i}$) and $2 \leq i \leq k$, there exists $u_{i} \in H_A^1(\G,\C)$ satisfying $E(u_{i},\G)=c_{\infty,i}$ with 
$c_{\infty,i-1}<c_{\infty,i}$ (see \eqref{monchain})
 and one of the following two cases must hold:
     \begin{enumerate}[label=\rm(\roman*),ref=\roman*]
        \item \label{pfthn41} either $u_i$ is a critical point of $E(\cdot, \G)$ constrained on $H_\mu(\G)$ with Lagrange multiplier $\omega_i \in [0,2c_{\infty,i}/\mu]\subset[0,\lambda_i]$
        \item \label{pfthn42} or $u_i$ is a critical point of $E(\cdot, \G)$ constrained on $H_\nu(\G)$ for some $0<\nu < \mu$ with Lagrange multiplier $\omega_i =0$. 
    \end{enumerate}
    Repeating the arguments of Lemma \ref{lemnonex} for $c=\lambda_k/2$, we obtain that, if $\mu$ is small enough ($\mu\in(0,\mu^*_k)$ with $\mu^*_k:=\min\{\mu_{\lambda_{k}/2,p}, \widetilde{\mu}_{k},\mu^{**}_k\}$), for any $1\leq i\leq k$, (\ref{pfthn42}) cannot occur. Thus, 
for any $0<\mu<\mu^*_k$, $E( \cdot,\G)$ has at least $k$ distinct critical points $u_1, u_2,\ldots,u_k$ constrained on $H_\mu(\G)$, which completes the proof of Theorem \ref{thn4}.
\end{proof}

\appendix
\section{Regularity of the critical points of \texorpdfstring{$E$}{E}}\label{regularity}

In this appendix we give some details about the regularity of our solutions.

Assume that $A$ is continuously differentiable on every edge $e\in \E$ and $V$ is continuous on every edge $e\in \E$. For any compact graph $\G$ or noncompact graph $\G$ having a non-empty compact core $\K$, if $u\in H_A^1(\G,\C)$ is a critical point of $E( \cdot,\G)$ constrained on $H_\mu(\G)$, then $u$ is twice continuously differentiable on every edge $e\in \E$ and there exists a Lagrange multiplier $\lambda \in \R$ such that $u$ is a solution of equation \eqref{eqmagcompact} or equation \eqref{eqmagloc}, respectively.

Indeed, let $u\in H_A^1(\G,\C)$ be a critical point of $E( \cdot,\G)$ constrained on $H_\mu(\G)$ with Lagrange multiplier $\lambda \in \R$. 
For any edge $e \in \E$ and bounded interval $I\subset I_e$,  we have
\begin{align*}
\int_{I} \abs{u_e'}^2+\abs{u_e}^2\,dx
&\leq 2\int_{I}\abs{D_{A_e}u_e}^2\,dx +2\int_{I}\abs{A_eu_e}^2\,dx +  \int_{I}V_e(x)\abs{u_e}^2\,dx\\
&\leq2 \int_{I} \abs{D_{A_e}u_e}^2\,dx +2\norm{A_e}^2_{L^\infty(I,\R)}\norm{u_e}_{L^2(I,\C)}^2+  \int_{I} V_e(x)\abs{u_e}^2\,dx
\end{align*}
and thus, $u_e \in H^1(I,\C)$. Since $u$ is a constrained critical point of $E( \cdot,\G)$,  for any $\varphi \in C_0^\infty(I,\C)$, integrating by parts (see, e.g.,\cite[Corollary 8.10]{Ha}), we obtain
 \begin{equation*}%
      \int_{I} u_e'\bar{\varphi}'+2iA_e(x)u_e'\bar{\varphi} + iA_e'(x)u_e\bar{\varphi}+(\abs{A_e(x)}^2+V_e(x))u_e\bar{\varphi} \,dx=\int_{I}\lambda u_e\bar{\varphi}+\chi_{\K,e}\abs{u_e}^{p-2}u_e\bar{\varphi} \,dx,
 \end{equation*}
 where $\chi_{\K,e}\equiv 1$ if $e\in \E$ is a bounded edge and $\chi_{\K,e}\equiv 0$ if $e\in \E$ is an unbounded edge.
 Then, by the definition of weak derivative, 
 \begin{equation}\label{eqderi}
    u_e''=
 -\lambda u_e
 -\chi_{\K,e}\abs{u_e}^{p-2}u_e
 +2iA_e(x)u_e'
 + iA_e'(x)u_e
 +(\abs{A_e(x)}^2+V_e(x))u_e \quad 
 \text{on }I.
 \end{equation}
 Hence, $u_e'' \in L^2(I,\C)$, and it follows from \cite[Theorem 8.2]{Ha} that $u_e' \in C(I,\C)$. Then, by \cite[Remark 6 in Chapter 8]{Ha} and \eqref{eqderi}, we conclude that $u_e \in C^2(I,\C)$. Since $e \in \E$ and $I\subset I_e$ are arbitrary, $u_e$ is twice continuously differentiable for every edge $e\in \E$.

Finally, fix $v\in\V$ and choose $\varepsilon>0$ such that
$2\varepsilon<\ell_e$ for every bounded edge $e\succ v$.
Set
$$
B_\varepsilon(v):=\{x\in\G:\operatorname{dist}(x,v)<\varepsilon\}.
$$
Then $B_\varepsilon(v)\cap\V=\{v\}$.
Take $\varphi\in H^1_{A,c}(\G,\C)$ with
$\operatorname{supp}\varphi\subset B_\varepsilon(v)$.
Integrating by parts and using \eqref{eqderi}, we obtain
\begin{equation*}
\begin{aligned}
0
&=\langle E'(u,\G),\varphi\rangle-\lambda(u,\varphi)_2\\
&=\operatorname{Re}\sum_{e\in\E}\int_{I_e}
\Big(
D_{A_e}u_e\,\overline{D_{A_e}\varphi_e}
+(V_e-\lambda)u_e\overline{\varphi_e}
-\chi_{\K,e}\abs{u_e}^{p-2}u_e\overline{\varphi_e}
\Big)\,dx\\
&=\operatorname{Re}\sum_{e\in\E}\int_{I_e}
\Big[-u_e''+2iA_eu_e'+iA_e'u_e
+(A_e^2+V_e-\lambda)u_e
-\chi_{\K,e}\abs{u_e}^{p-2}u_e\Big]
\overline{\varphi_e}\,dx\\
&\quad-\operatorname{Re}\left[
i\overline{\varphi(v)}
\sum_{e\succ v}
\left(\frac{1}{i}u'_e(v)-A_e^\pm(v)u_e(v)\right)
\right]\\
&=-\operatorname{Re}\left[
i\overline{\varphi(v)}
\sum_{e\succ v}
\left(\frac{1}{i}u'_e(v)-A_e^\pm(v)u_e(v)\right)
\right].
\end{aligned}
\end{equation*}
Since $\varphi(v)$ can be chosen arbitrarily in $\C$, we conclude that
$$
\sum_{e\succ v}
\left(\frac{1}{i}u'_e(v)-A_e^\pm(v)u_e(v)\right)=0.
$$
As $v\in\V$ is arbitrary, this proves the magnetic Kirchhoff conditions in \eqref{eqmagcompact} and \eqref{eqmagloc}.
For the whole-graph nonlinearity in \eqref{eqmagnoncompact}, the same argument applies with $\chi_{\K,e}\equiv1$ on every edge. Thus, the corresponding constrained critical points also satisfy the edgewise equations and the magnetic Kirchhoff conditions.

\section*{Acknowledgements}
The authors would like to thank the anonymous referees for their valuable comments and suggestions. Pietro d'Avenia is a member of GNAMPA (INdAM) and is partially supported by the GNAMPA project {\em Problemi di ottimizzazione in PDEs da modelli biologici} (CUP E5324001950001). He has been financed by European Union - Next Generation EU - PRIN 2022 PNRR P2022YFAJH {\em Linear and Nonlinear PDE's: New directions and Applications} and is also supported by the Italian Ministry of University and Research under the Program Department of Excellence L. 232/2016 (CUP D93C23000100001). Chao Ji is supported by the National Natural Science Foundation of China (No.12571117).

\end{document}